\RequirePackage{fix-cm}
\documentclass[smallextended]{svjour3}       % onecolumn (second format)
\smartqed  % flush right qed marks, e.g. at end of proof
\usepackage{graphicx}
\usepackage[toc,page]{appendix}

\usepackage{bm}
\usepackage[version=4]{mhchem}
\usepackage{siunitx}
\usepackage{longtable,tabularx}
\usepackage{xspace}

\usepackage{booktabs}
\usepackage{subcaption} 

\usepackage{graphicx}

\usepackage{float}
\floatstyle{plaintop}
\restylefloat{table}

\usepackage{tikz}

\usepackage{subcaption}

\usepackage{hyperref}

\usepackage{cancel}

\usepackage{algorithm,algpseudocode}

\newcommand{\Table}[1]{\mbox{Table~\ref{#1}}}

\newcommand{\Fig}[1]{\mbox{Fig.~\ref{#1}}}
\newcommand{\Figs}[2]{\mbox{Figs.~\ref{#1}} and \ref{#2}}

\newcommand{\Prop}[1]{\mbox{Proposition~\ref{#1}}}
\newcommand{\Props}[2]{\mbox{Propositions~\ref{#1}} and \ref{#2}}

\newcommand{\Lem}[1]{\mbox{Lemma~\ref{#1}}}
\newcommand{\Lems}[2]{\mbox{Lemmas~\ref{#1}} and \ref{#2}}

\newcommand{\Sec}[1]{\mbox{Section~\ref{#1}}}
\newcommand{\Secs}[2]{\mbox{Sections~\ref{#1}} and \ref{#2}}

\newcommand{\Eq}[1]{\mbox{Eq.~(\ref{#1})}}
\newcommand{\Eqs}[2]{\mbox{Eqs.~(\ref{#1})} and (\ref{#2})}
\newcommand{\Eqss}[3]{\mbox{Eqs.~(\ref{#1})}, (\ref{#2}), and (\ref{#3})}

\newcommand{\bsym}[1]{\boldsymbol{#1}}
\newcommand{\mat}[1]{\ensuremath{\mathsf{#1}}}

\newcommand{\Idty}[0]{\mat{I}} % identity

\definecolor{Red}{RGB}{235,40,25}
\definecolor{Green}{rgb}{0,1,0}

\newcommand\p[2]{\frac{\partial #1}{\partial #2}}
\newcommand\pp[3]{\frac{\partial^2 #1}{\partial #2 \partial #3}}

\DeclareMathOperator{\mydiag}{diag}

\newcommand{\trace}[1]{\text{tr} \left( #1 \right) }

\newcommand{\argmax}[0]{\ensuremath{\operatornamewithlimits{argmax}}}

\newcommand{\eigmin}[0]{\ensuremath{\lambda_{\text{min}}}}
\newcommand{\eigmax}[0]{\ensuremath{\lambda_{\text{max}}}}

\newcommand{\ie}[0]{i.e.\ }

\usepackage{amssymb}% http://ctan.org/pkg/amssymb
\usepackage{pifont}% http://ctan.org/pkg/pifont
\newcommand{\cmark}[0]{\textcolor{green}{\ding{51}} }%
\newcommand{\xmark}[0]{\textcolor{red}{\ding{55}} }%

\newcommand{\ignore}[1]{} % comment out large sections of code

\newcommand\yesnumber{\addtocounter{equation}{1}\tag{\theequation}}

\newcommand{\RNum}[1]{\uppercase\expandafter{\romannumeral #1\relax}}

\newcommand{\zero}[0]{\ensuremath{\textbf{0}}}
\newcommand{\one}[0]{\ensuremath{\textbf{1}}}

\newcommand{\fvec}[0]{\ensuremath{\bsym{f}}}

\newcommand{\kvec}[0]{\ensuremath{\bsym{k}}}
\newcommand{\rvec}[0]{\ensuremath{\bsym{r}}}

\newcommand{\wvec}[0]{\ensuremath{\bsym{w}}}

\newcommand{\xvec}[0]{\ensuremath{\bsym{x}}}
\newcommand{\yvec}[0]{\ensuremath{\bsym{y}}}

\newcommand{\thetavec}[0]{\ensuremath{\bsym{\theta}}}
\newcommand{\gammavec}[0]{\bsym{\gamma}}

\newcommand{\A}[0]{\mat{A}}

\newcommand{\K}[0]{\mat{K}}

\newcommand{\R}[0]{\mat{R}}

\newcommand{\X}[0]{\mat{X}}

\newcommand{\Lmat}[0]{\mat{L}}
\newcommand{\Ltil}[0]{\tilde{\mat{L}}}

\newcommand{\Rtil}[0]{\tilde{\mat{R}}}

\newcommand{\fgrad}{\ensuremath{\bsym{f}_{\nabla}}}

\newcommand{\etaK}{\ensuremath{\eta_{\K}}}
\newcommand{\etaKg}{\ensuremath{\eta_{\Kg}}}
\newcommand{\etaKgtil}{\ensuremath{\eta}_{\Kgtil}}

\newcommand{\kvecgrad}{\ensuremath{\bsym{k}_\nabla}}

\newcommand{\Kg}[0]{\ensuremath{\mat{K}_{\nabla}}}

\newcommand{\condmax}[0]{\kappa_{\max}}

\newcommand{\bSigmagtil}{\ensuremath{\kappa(\Sigma(\gammavec)) < \condmax \, \forall \, \gammavec > 0}}

\newcommand\vmin{\ensuremath{v_{\text{min}}}}

\newcommand\vreq{\ensuremath{v_{\text{req}}}}
\newcommand\vminset{\ensuremath{v_{\text{min,set}}}}

\newcommand{\Pmat}[0]{\mat{P}}
\newcommand{\Pinv}[0]{\mat{P}^{-1}}

\newcommand{\rtil}[0]{\ensuremath{\tilde{r}}}
\newcommand{\xtil}[0]{\ensuremath{\tilde{x}}}
\newcommand{\ytil}[0]{\ensuremath{\tilde{y}}}

\newcommand{\rvectil}[0]{\ensuremath{\tilde{\rvec}}}
\newcommand{\xvectil}[0]{\ensuremath{\tilde{\xvec}}}

\newcommand{\Kgtil}[0]{\ensuremath{\tilde{\mat{K}}_{\nabla}}}

\newcommand{\Sigmag}[0]{\ensuremath{\Sigma_{\nabla}}}
\newcommand{\Sigmagtil}[0]{\ensuremath{\tilde{\Sigma}_{\nabla}}}

\newcommand{\onemod}[0]{\ensuremath{\hat{\one}}}
\newcommand{\onebar}[0]{\ensuremath{\bar{\one}}}

\newcommand{\sigK}{\ensuremath{\sigma_{\K}}}
\begin{document}

\title{A Solution to the Ill-Conditioning of Gradient-Enhanced Covariance Matrices for Gaussian Processes %\thanks{Grants or other notes
}

\titlerunning{A solution for ill-conditioned covariance matrices} % if too long for running head

\author{Andr\'e L.\ Marchildon        
		 \and David W. Zingg 
}


\institute{Andr\'e L.\ Marchildon \at
              University of Toronto Institute for Aerospace Studies, Toronto, ON, Canada \\
              % Tel.: +123-45-678910\\
              % Fax: +123-45-678910\\
              \email{andre.marchildon@mail.utoronto.ca}           %  \\
           \and
           David W.\ Zingg \at
           University of Toronto Institute for Aerospace Studies, Toronto, ON, Canada
}

\date{Received: date / Accepted: date}

\maketitle

\begin{abstract}

%Insert your abstract here. Include keywords, PACS and mathematical
%subject classification numbers as needed.

Gaussian processes provide probabilistic surrogates for various applications including classification, uncertainty quantification, and optimization. Using a gradient-enhanced covariance matrix can be beneficial since it provides a more accurate surrogate relative to its gradient-free counterpart. An acute problem for Gaussian processes, particularly those that use gradients, is the ill-conditioning of their covariance matrices. Several methods have been developed to address this problem for gradient-enhanced Gaussian processes but they have various drawbacks such as limiting the data that can be used, imposing a minimum distance between evaluation points in the parameter space, or constraining the hyperparameters. In this paper a new method is presented that applies a diagonal preconditioner to the covariance matrix along with a modest nugget to ensure that the condition number of the covariance matrix is bounded, while avoiding the drawbacks listed above. Optimization results for a gradient-enhanced Bayesian optimizer with the Gaussian kernel are compared with the use of the new method, a baseline method that constrains the hyperparameters, and a rescaling method that increases the distance between evaluation points. The Bayesian optimizer with the new method converges the optimality, \ie the $\ell_2$ norm of the gradient, an additional 5 to 9 orders of magnitude relative to when the baseline method is used and it does so in fewer iterations than with the rescaling method. The new method is available in the open source python library GpGradPy, which can be found at \url{https://github.com/marchildon/gpgradpy/tree/paper_precon}. All of the figures in this paper can be reproduced with this library.


%\subclass{60G15 \and 65K10}

\subclass{15A12, % Linear and multilinear algebra; matrix theory: Conditioning of matrices 
			\and 60G15, % Stochastic processes: Gaussian processes
			\and 65K99 % Numerical methods for mathematical programming, optimization and variational techniques; None of the above, but in this section
			}

% https://cran.r-project.org/web/classifications/MSC.html

\end{abstract}

%--------------------
% New section
%--------------------
\section{Introduction} \label{Sec_Intro}

In diverse fields and for various applications, such as uncertainty quantification, classification, and optimization, an expensive function of interest must be repeatedly evaluated \cite{zingg_comparative_2008,shahriari_taking_2016,schulz_tutorial_2018}. To minimize the computational cost it is desirable to minimize the number of expensive function evaluations. One way to achieve this is by constructing a surrogate that approximates the function of interest and is inexpensive to evaluate. Various methods to construct surrogates can be utilized, including using a Gaussian process (GP). This method provides a nonparametric probabilistic surrogate. The nonparametric component of the GP indicates that it does not depend on a parametric functional form, unlike a polynomial surrogate where the order of the basis function must be selected a priori. As for the probabilistic component of the GP, this enables the surrogate to provide an estimate for the function of interest and to quantify the uncertainty in its estimate \cite{rasmussen_gaussian_2006}.
A Gaussian process requires a mean and a covariance function \cite{ababou_condition_1994}. %\cite{ababou_condition_1994,jordan_prediction_1998-1}
A constant is often used for the former and its value is set by maximizing the marginal log-likelihood \cite{toal_kriging_2008,toal_adjoint_2009,toal_development_2011,ollar_gradient_2017}. For the covariance function, it is popular to use kernels, of which many are available \cite{davis_six_1997,rasmussen_gaussian_2006}. The most popular kernel is the Gaussian kernel, which is also known as the squared exponential kernel \cite{rasmussen_gaussian_2006,shahriari_taking_2016,wu_exploiting_2018}. The desirable properties of this kernel include its hyperparameters that can be tuned, its simplicity, and its smoothness. This final property enables the surrogate to be constructed using gradient evaluations, which makes the surrogate more accurate \cite{dalbey_efficient_2013,eriksson_scaling_2018,wu_exploiting_2018}. 

Gradient-enhanced GPs use both the value and gradient of the function of interest to construct the probabilistic surrogate. By using gradients with the GP, a more accurate surrogate is constructed that matches both the value and gradient of the function of interest where it has been evaluated in the parameter space \cite{osborne_gaussian_2009,ulaganathan_performance_2016,wu_bayesian_2017}. This is particularly useful in high-dimensional parameter spaces since a single gradient evaluation provides much more information than a single function evaluation. The gradient-enhanced covariance matrix can be constructed either with the direct method or the indirect method \cite{zimmermann_maximum_2013}. The former modifies the structure of the gradient-free covariance matrix while the latter does not. The direct method is much more common \cite{han_improving_2013,dalbey_efficient_2013,wu_exploiting_2018,laurent_overview_2019} and is used in this paper. A drawback of using gradient-enhanced GPs is that the covariance matrix is larger than its gradient-free counterpart and is thus more expensive to invert. Various strategies have been developed to mitigate this additional cost by using random Fourier features \cite{hung_random_2021}, or by exploiting the structure of the gradient-enhanced covariance matrix \cite{de_roos_high-dimensional_2021}.

A ubiquitous problem in the use of GPs is the ill-conditioning of their covariance matrices \cite{ababou_condition_1994,kostinski_condition_2000,zimmermann_condition_2015}. This problem is present with the use of many kernels, including the Gaussian kernel. Various factors have been identified that exacerbate the ill-conditioning, such as having the data points too close together \cite{davis_six_1997}. The ill-conditioning of the covariance matrix is problematic since it can cause the Cholesky factorization to fail \cite{higham_cholesky_2009}, and it also increases the numerical error. 
Adding a small positive nugget to the diagonal of the gradient-free covariance matrix is sufficient to ensure that the condition number of the matrix is below a user-set threshold \cite{mohammadi_analytic_2017}. %\cite{pepelyshev_role_2010,mohammadi_analytic_2017}.

The ill-conditioning of the gradient-enhanced covariance matrix is even more acute than the gradient-free case, and the addition of a nugget is insufficient on its own to alleviate this problem \cite{he_instability_2018,dalbey_efficient_2013}. Various approaches have been attempted to mitigate the ill-conditioning problem, such as removing certain data points until the condition number is sufficiently low \cite{march_gradient-based_2011,dalbey_efficient_2013}, or imposing a minimum distance constraint between data points in the parameter space \cite{osborne_gaussian_2009}. Both methods have significant drawbacks since they restrict the data available to construct the surrogate. Furthermore, neither method guarantees that the condition number of the covariance matrix remains below a user-set threshold as the hyperparameters are optimized. There is one recent method that does ensure that the condition number of the gradient-enhanced covariance matrix remains below a user-set threshold when the Gaussian kernel is used \cite{marchildon_non-intrusive_2023}. This method uses non-isotropic rescaling of the data in order to have a set minimum distance between the data points. While data points cannot be collocated, they can get arbitrarily close in the parameter space, and the method allows all of the data points to be kept in the construction of the gradient-enhanced covariance matrix. However, the drawback of this method is that, in some cases, the rescaling needs to be done iteratively, which requires the hyperparameters to be optimized again. This adds additional complexity and computational cost.

The new method presented in this paper shares the same benefits as the rescaling method from \cite{marchildon_non-intrusive_2023}, \ie all of the data points can be used, there is no minimum distance constraint between the data points in the parameter space, and the condition number of the gradient-enhanced covariance matrix is bounded. The new method also has two additional benefits: it only requires a single optimization of the hyperparameters, \ie it is not iterative, and there is no need for a constraint on the condition number for the optimization of the hyperparameters. This simplifies the implementation of the new method and reduces its computational cost. 

The new and rescaling methods are available in the open source python library GpGradPy, which can be accessed at \url{https://github.com/marchildon/gpgradpy/tree/paper_precon}. This library contains the Gaussian, Mat\'ern $\frac{5}{2}$, and rational quadratic kernels.

The notation used in this paper is presented in \Sec{Sec_Notation}. In \Sec{Sec_GP} the GP is presented along with the Gaussian kernel and the covariance matrix. A modified covariance matrix is derived in \Sec{Sec_GpNew}. In \Sec{Sec_WellCondMtd} it is demonstrated how the condition number of the modified covariance matrix can be bounded with the use of a nugget. Details on the implementation of the new covariance matrix can be found in \Sec{Sec_Implementation} and optimization results are provided in \Sec{Sec_Results}. Finally, the conclusions of the paper are presented in \Sec{Sec_Conclusions}. 

%--------------------
% New section
%--------------------
\section{Notation} \label{Sec_Notation}

Sans-serif capital letters are used for matrices. For example, $\Idty$ is the identity matrix, and $\X$ is an $n_x \times d$ matrix that holds the location of $n_x$ evaluation points in a $d$ dimensional parameter space. Vectors are denoted in lowercase bold font. For instance, $\xvec$ and $\yvec$ are vectors of length $d$ denoting arbitrary points in the parameter space. The $i$-th row of $\X$ is denoted as $\xvec_{i:}$ and its $j$-th column is indicated as $\xvec_{:j}$. Finally, scalars are denoted in lowercase letters such as $x_{ij}$, which is the entry at the $i$-th row and $j$-th column of $\X$. The symbols $\zero_d$ and $\one_d$ are vectors of length $d$ with all of their entries equal to zero and one, respectively. Variations of these symbols such as $\bar{\one}$ or $\onemod$ are used to indicate a matrix of ones or a vector where some of its entries are zero, respectively. These will be clarified when they appear in the paper.

%--------------------
% New section
%--------------------
\section{Gaussian process} \label{Sec_GP}

%--------------------
% New subsection
%--------------------
\subsection{Gradient-free covariance matrix} \label{Sec_GP_woGrad}

To fully define a GP we require a mean function and a covariance function. The mean function is selected here to be the constant $\beta$, which is a hyperparameter that is selected by maximizing the marginal log-likelihood function that will be presented in \Sec{Sec_GP_Eval}. For the covariance function we use the popular Gaussian kernel
\begin{align*}
	k(\xvec, \yvec; \gammavec) 
		%&= k(\rvec) \\
		%&= e^{-\sum_{i=1}^d \theta_i r_i^2} \\
		&= e^{-\frac{1}{2}\sum_{i=1}^d \gamma_i^2 (x_i - y_i)^2}, \yesnumber \label{Eq_Gaussian_kernel}
\end{align*}
where $\gamma_i > 0 \, \forall \, i \in \{1, \ldots, d \}$ are hyperparameters. The Gaussian kernel is typically presented with $\thetavec = \gammavec^2 / 2$ as its hyperparameters but it is simpler in the later derivations to use $\gammavec$ instead. The Gaussian kernel is a stationary kernel since it depends only on $\rvec = \xvec - \yvec$, \ie the relative location of $\yvec$ to $\xvec$. The gradient-free Gaussian kernel matrix is
\begin{equation} \label{Eq_matrix_K}
	\K = \K(\X; \gammavec) =
	\begin{bmatrix}
		1 	& k(\xvec_{1:}, \xvec_{2:},; \gammavec) 		& \ldots 	& k(\xvec_{1:}, \xvec_{n_x:}; \gammavec) \\
		k(\xvec_{2:}, \xvec_{1:}; \gammavec) 	& 1 		& \ldots 	& k(\xvec_{2:}, \xvec_{n_x:}; \gammavec) \\
		\vdots 	&	\vdots						& \ddots 	& \vdots \\
		k(\xvec_{n_x:}, \xvec_{1:}; \gammavec) 	& k(\xvec_{n_x:}, \xvec_{2:}; \gammavec) 	& \ldots 	& 1,
	\end{bmatrix},
\end{equation}
where its diagonal entries are all unity. In general, the $i$-th diagonal entry of $\K$ is $k(\xvec_{i:}, \xvec_{i:}; \gammavec)$. The gradient-free Gaussian kernel matrix $\K$ is also a correlation matrix since it satisfies all of the properties of the following definition.

\begin{definition} \label{Def_corr_matrix}
	A correlation matrix must satisfy all of the following conditions:
	\begin{enumerate}
		\item All of the entries in the square matrix are real and between $-1$ and $1$
		\item The diagonal entries of the matrix are all unity
		\item The matrix is positive semidefinite
	\end{enumerate}
\end{definition}

The noise-free regularized gradient-free covariance matrix is given by 
\begin{equation}
	\Sigma(\X; \gammavec, \etaK) 	= \sigK^2 \left( \K(\X; \gammavec) + \etaK \Idty \right),
\end{equation}
where the hyperparameter $\sigK^2$ is the variance of the stationary residual error, and the nugget $\etaK$ is used to have $\kappa(\Sigma) \leq \condmax$, where $\condmax > 1$ is the maximum allowed condition number and is set by the user. The nugget is discussed further in \Sec{Sec_WellCondMtd_Nugget}.

%--------------------
% New subsection
%--------------------
\subsection{Gradient-enhanced covariance matrix} \label{Sec_GP_wGrad}

Constructing the gradient-enhanced kernel matrix requires the derivatives of the kernel with respect to its inputs:
\begin{align}	
	\p{k(\xvec, \yvec)}{x_i}
		&= -\gamma_i^2 \left(x_i - y_i \right) k(\xvec, \yvec) \label{Eq_Kg_entry_grad_w_obj} \\
	\p{k(\xvec, \yvec)}{y_j} 
		&= \gamma_j^2 \left(x_j - y_j \right) k(\xvec, \yvec) \label{Eq_Kg_entry_obj_w_grad} \\
	\p{^2 k(\xvec, \yvec)}{x_i \partial y_j} 
		&= \left( \delta_{ij} \gamma_i^2 - \gamma_i^2 \gamma_j^2 \left(x_i - y_i \right) \left(x_j - y_j \right) \right) k(\xvec, \yvec), \label{Eq_Kg_entry_grad_w_grad}
\end{align}
where $\delta_{ij}$ is the Kronecker delta. The gradient-enhanced kernel matrix is given by
\begin{align} 
\Kg(\X; \gammavec) &= 
\begin{bmatrix}
	\K 				& \p{\K}{y_1} 				& \ldots 	& \p{\K}{y_d} \\
	\p{\K}{x_1} 	& \pp{\K}{x_1}{y_1} 	& \ldots 	& \pp{\K}{x_1}{y_d} \\ 
	\vdots 			& 	\vdots					& \ddots 	& \vdots \\
	\p{\K}{x_d} 	& \pp{\K}{x_d}{y_1} 	& \ldots 	& \pp{\K}{x_d}{y_d}
\end{bmatrix} \label{Eq_Kg_form} \\
&=
\begin{bmatrix}
	\K & \gamma_1^2 \R_1 \odot \K & \ldots & \gamma_d^2 \R_d \odot \K \\ 
	-\gamma_1^2 \R_1 \odot \K & \left( \gamma_1^2 \onebar - \gamma_1^4 \R_1^{(2)} \right) \odot \K & \ldots & -\gamma_1^2 \gamma_d^2 \R_1 \odot \R_d \odot \K \\
	\vdots & \vdots & \ddots & \vdots	\\
	-\gamma_d^2 \R_d \odot \K & - \gamma_1^2 \gamma_d^2 \R_1 \odot \R_d \odot \K & \ldots & \left( \gamma_d^2 \onebar - \gamma_d^4 \R_d^{(2)} \right) \odot \K \\
\end{bmatrix}, \label{Eq_Kg}
\end{align}
where $\onebar$ is a matrix of ones of size $n_x \times n_x$, 
the operator $\odot$ is the Hadamard product for elementwise multiplication, and $\R_i$ is a skew-symmetric matrix given by
\begin{align}
	\R_i(\X) 
		&= \xvec_{:i} \one_{n_x}^\top - \one_{n_x} \xvec_{:i}^\top \nonumber \\
		&=
	\begin{bmatrix}
		0 & x_{1i} - x_{2i} & \ldots & x_{1i} - x_{n_xi} \\
		x_{2i} - x_{1i} & 0 & \ldots & x_{2i} - x_{n_xi} \\
		\vdots & \vdots & \ddots & \vdots \\
		x_{n_xi} - x_{1i} & x_{n_xi} - x_{2i} & \ldots & 0
	\end{bmatrix}. \label{Eq_matR}
\end{align}
Unlike $\K$, $\Kg$ is not a correlation matrix since it does not satisfy the first and second conditions in Definition \ref{Def_corr_matrix}. This is clear from checking the diagonal of $\Kg$:
\begin{equation} \label{Eq_diag_Kg}
	\mydiag(\Kg) = [\underbrace{1, \ldots, 1}_{n_x}, \underbrace{\gamma_1^2, \ldots \gamma_1^2}_{n_x}, \ldots, \underbrace{\gamma_d^2, \ldots \gamma_d^2}_{n_x}].
\end{equation} 
Definition \ref{Def_corr_matrix} would only be satisfied for $\Kg$ if $\gamma_1 = \ldots = \gamma_d = 1$. However, the hyperparameters $\gammavec$ are set by maximizing the marginal log-likelihood, which is introduced in the following subsection. The noise-free regularized gradient-enhanced covariance matrix is
\begin{equation} \label{Eq_grad_cov_mat}
	\Sigmag(\gammavec; \etaKg) = \sigK^2 \left( \Kg(\gammavec) + \etaKg \Idty \right),
\end{equation}
where the nugget $\etaKg$ is a regularization term that is used to ensure that $\kappa(\Sigmag) \leq \condmax$, as detailed in \Sec{Sec_WellCondMtd}.

%--------------------
% New subsection
%--------------------
\subsection{Evaluating the Gaussian process} \label{Sec_GP_Eval}

The mean and variance of the gradient-enhanced Gaussian process are evaluated with \cite{wu_bayesian_2017}
\begin{align}
	\mu_f(\xvec) 
		&= \beta + \kvecgrad^\top (\xvec) \left( \Kg + \etaKg \Idty \right)^{-1} \left(\fgrad - \beta \onemod \right) \label{Eq_surr_mu} \\
	\sigma_f^2(\xvec) 
		&= \sigK^2 \left( k(\xvec, \xvec) - \kvecgrad^\top (\xvec) \left( \Kg + \etaKg \Idty \right)^{-1} \kvecgrad(\xvec) \right), \label{Eq_surr_sig} 
\end{align}
where $\onemod = [\one_{n_x}^\top, \zero_{n_x d}^\top]^\top$, and
%where
%
\begin{align} 
	\kvecgrad(\xvec; \X) =
	\begin{bmatrix}
		\kvec(\X, \xvec) \\
		\p{\kvec(\X, \xvec)}{x_1} \\
		\vdots \\
		\p{\kvec(\X, \xvec)}{x_d}
	\end{bmatrix}, \quad
	\fgrad(\X) = 
	\begin{bmatrix}
		\fvec(\X) \\
		\p{\fvec(\X)}{x_1} \\
		\vdots \\
		\p{\fvec(\X)}{x_d}
	\end{bmatrix}, \label{Eq_def_kvecgrad_fgrad}
\end{align}
where $\fvec(\X)$ is the function of interest evaluated at all of the rows in $\X$. In this paper the gradient of the function of interest is calculated analytically, but it could also be calculated with algorithmic differentiation or approximated with finite differences.

Prior to evaluating the GP, its hyperparameters must first be selected, which is commonly done by maximizing the marginal likelihood
\cite{toal_kriging_2008,toal_adjoint_2009,toal_development_2011,ollar_gradient_2017}
\begin{align*} %\label{Eq_log_lkd}
	L(\gammavec, \beta, \sigK^2; \X, \fgrad, \etaKg) 
		&= \frac{e^{-\frac{ \left(\fgrad - \beta \hat{\one} \right)^\top \Sigmag^{-1} \left( \fgrad - \beta \hat{\one} \right) }{2}} }{ \left(2 \pi \right)^{\frac{n_x(d+1)}{2}} \sqrt{\det \left(\Sigmag \right) }} \\
		&= \frac{e^{-\frac{ \left(\fgrad - \beta \hat{\one} \right)^\top \left( \Kg + \etaKg \Idty \right)^{-1} \left(\fgrad - \beta \hat{\one} \right) }{ 2 \sigK^2 }} }{ \left(2 \pi \sigK^2 \right)^{\frac{n_x(d+1)}{2}} \sqrt{ \det \left( \Kg + \etaKg \Idty \right) }}. % \\
%	\ln(L(\gammavec, \beta, \sigK^2; \X, \fgrad, \etaKg)) 
%		&= -\frac{n_x(d+1) \left( \ln(2 \pi) + \ln( \sigK^2) \right) + \ln \left( \det \left( \Kg + \etaKg \Idty \right) \right)}{2} \\
%		& \quad \quad - \frac{ \left(\fgrad - \beta \hat{\one}\right)^\top \left( \Kg + \etaKg \Idty \right)^{-1} \left(\fgrad - \beta \hat{\one} \right) }{ 2 \sigK^2}.
\end{align*}
For the noise-free case being considered it is straightforward to get closed-form solutions for $\beta$ and $\sigK^2$ that maximize $\ln(L)$ \cite{toal_kriging_2008}:
\begin{align}
	\beta(\gammavec; \X, \fgrad, \etaKg) 
		&= \frac{ \hat{\one}^\top \left( \Kg + \etaKg \Idty \right)^{-1} \fgrad}{ \onemod^\top \left( \Kg + \etaKg \Idty \right)^{-1} \onemod} \label{Eq_hp_mu} \\
	\sigK^2(\gammavec; \X, \fgrad, \etaKg, \beta) 
		&= \frac{ \left(\fgrad - \beta \onemod \right)^\top \left( \Kg + \etaKg \Idty \right)^{-1} \left(\fgrad - \beta \onemod \right) }{n_x(d+1)}. \label{Eq_hp_sigK}
\end{align}
Substituting these solutions for $\beta$ and $\sigK^2$ into $\ln(L)$ and dropping the constant terms gives

\begin{equation} \label{Eq_ln_lkd_final}
	\ln(L(\gammavec; \X, \etaKg, \sigK)) 
		= -\frac{ n_x(d+1) \ln(\sigK^2) + \ln(\det(\Kg + \etaKg \Idty)) }{2}. 
\end{equation}
The hyperparameters $\gammavec$ are set by maximizing \Eq{Eq_ln_lkd_final} numerically with the bound $\gammavec > 0$.

%--------------------------------------------------------------------------
\begin{figure}[t]
	\centering
	\begin{subfigure}[t]{0.49\textwidth}
	\includegraphics[width=\textwidth]{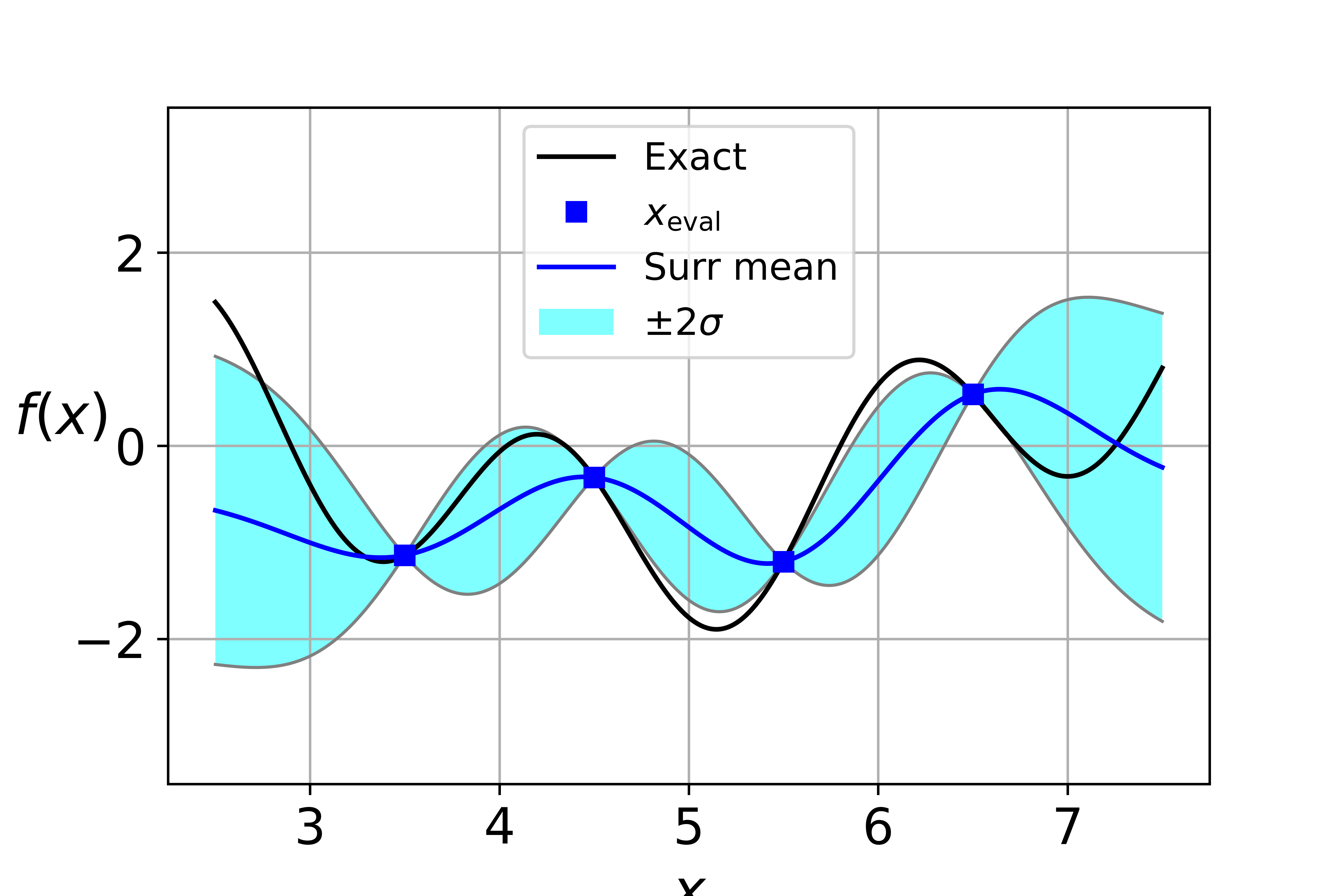}
	\caption{Gradient-free GP}
	\label{Fig_surr_1d_wo_grad}
	\end{subfigure}
	\hspace{0.05cm}
	\begin{subfigure}[t]{0.49\textwidth}
	\includegraphics[width=\textwidth]{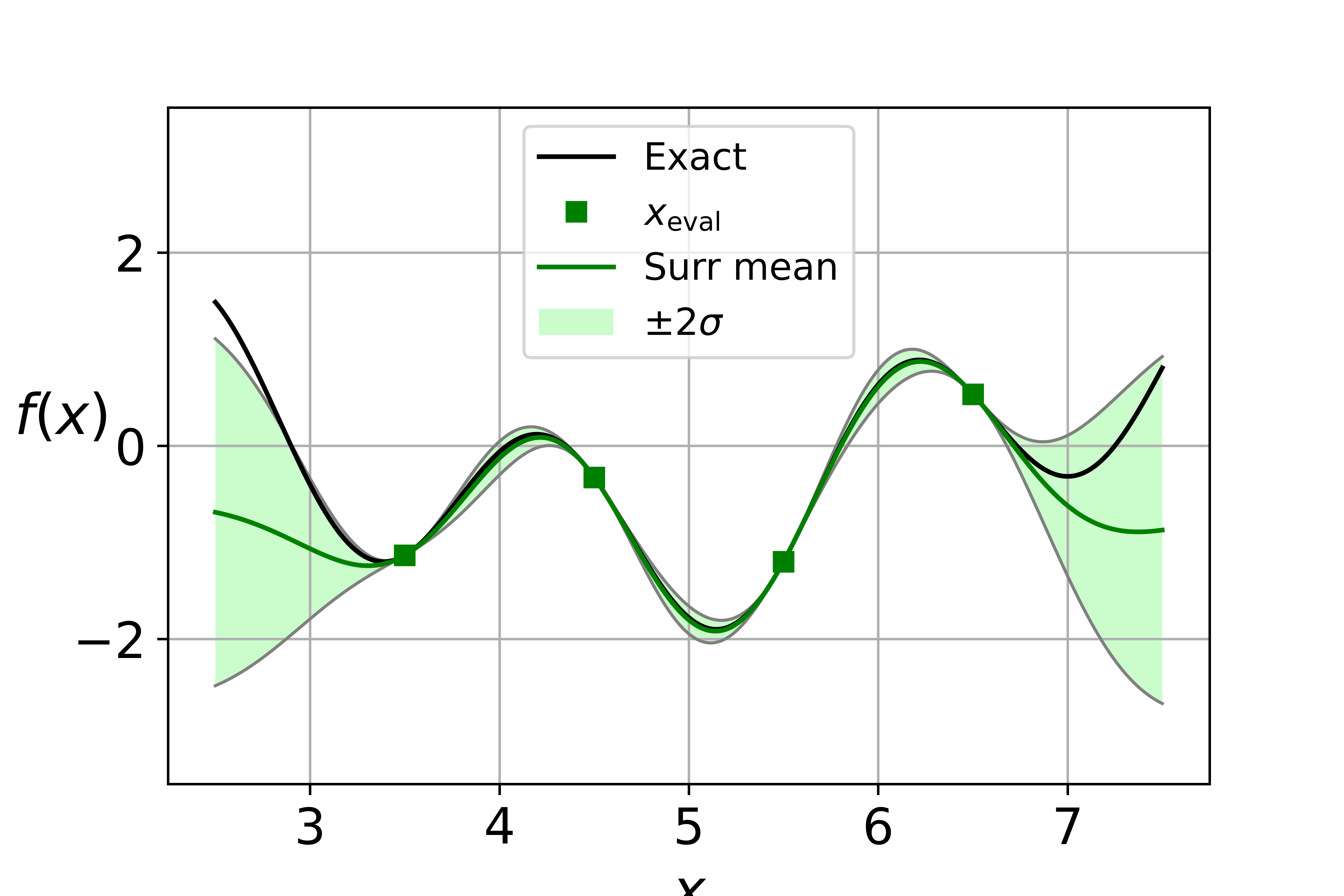}
	\caption{Gradient-enhanced GP}
	\label{Fig_surr_1d_w_grad}
	\end{subfigure}
	\caption{GPs with and without gradients that are approximating the function from \Eq{Eq_obj_surr_1d} with $\beta = -0.62$, $\sigK^2 = 1.07$, and $\gamma = 1.7$.}
	\label{Fig_surr_1d}
\end{figure}
%--------------------------------------------------------------------------

To highlight the benefit of using gradients we compare a gradient-free and a gradient-enhanced GP that approximate the following function:
\begin{equation} \label{Eq_obj_surr_1d}
	f(x) = \sin(x) + \sin \left( \frac{10x}{3} \right),
\end{equation}
which was evaluated at $\xvec = [3.5, 4.5, 5.5, 6.5]^\top$. Maximizing \Eq{Eq_ln_lkd_final} numerically for the gradient-enhanced GP provides the following hyperparameters: $\beta = -0.62$, $\sigK^2 = 1.07$, and $\gamma = 1.7$. The gradient-free and gradient-enhanced GPs can be seen in \Figs{Fig_surr_1d_wo_grad}{Fig_surr_1d_w_grad}, respectively. For \Fig{Fig_surr_1d_w_grad} the black line is \Eq{Eq_obj_surr_1d}, the solid green line is the mean of the surrogate from \Eq{Eq_surr_mu}, and the light green area represents $\pm 2 \sigma_f(x)$, where $\sigma_f$ comes from \Eq{Eq_surr_sig}. For \Fig{Fig_surr_1d_wo_grad} the mean and standard deviation of the gradient-free GP are calculated with equations analogous to \Eqs{Eq_surr_mu}{Eq_surr_sig} that omit the gradient evaluations and use the gradient-free kernel matrix.

It is clear from \Fig{Fig_surr_1d} that the use of gradients to construct the GP significantly improves its accuracy and also reduces its uncertainty, \ie $\sigma_f$. The benefit of using gradients to construct a GP is even greater for higher-dimensional parameter spaces since the gradient provides more information as the number of parameters increases. However, a significant problem for gradient-enhanced GPs is that their covariance matrices becomes extremely ill-conditioned, which is addressed in the following section.

%--------------------
% New section
%--------------------
\section{Modified gradient-enhanced covariance matrix} \label{Sec_GpNew}

A modified gradient-enhanced kernel matrix $\Kgtil$ is derived that can be regularized with a modest nugget such that its condition number is bounded below the user-set threshold $\condmax$. From \Eq{Eq_grad_cov_mat} the noise-free and unregularized gradient-enhanced covariance matrix for the vector $\fgrad$ from \Eq{Eq_def_kvecgrad_fgrad} is $\Sigmag = \sigK^2 \Kg$. A correlation matrix can be formed from a covariance matrix by normalizing by the standard deviation of the random variables, \ie $\fgrad$, which are the square root of the values along the diagonal of $\Sigmag$. Our gradient-enhanced correlation matrix $\Kgtil$ is thus given by
\begin{align}
	\Kgtil 
		&= \left(\frac{1}{\sigK} \Pinv \right)^{-1} \Sigmag \left(\frac{1}{\sigK} \Pinv \right)^{-1} \nonumber \\
		&= \Pinv \Kg \Pinv,
\end{align}
where 
\begin{align}
	\Pmat
		&= \mydiag \left( \sqrt{ \mydiag \left( \Kg \right)} \right) \label{Eq_Pmat_gen} \\
		&= \mydiag ( \one_{n_x}, \underbrace{\gamma_1, \ldots, \gamma_1}_{n_x}, \ldots, \underbrace{\gamma_d, \ldots, \gamma_d}_{n_x} ). \label{Eq_Pmat}
\end{align}

From \Eq{Eq_Kg}, $\Kg$ is formed with the block matrices composed of $\K$, as well as its first and second derivatives with respect to the entries of the dimensional vectors $\xvec$ and $\yvec$ from \Eq{Eq_Gaussian_kernel}. Similarly $\Kgtil$ is formed with block matrices composed of $\K$ as well as its first and second derivatives with respect to the dimensionless variables $\xtil_i = \gamma_i x_i$ and $\ytil_i = \gamma_i y_i$ for $i \in \{1, \ldots, d \}$. Consider the following chain rule
\begin{align*}
	\p{k(\xvec, \yvec; \gammavec)}{\xtil_i} 
		&= \p{k(\xvec, \yvec; \gammavec)}{x_i} \p{x_i}{\xtil_i} \\
		&= \p{k(\xvec, \yvec; \gammavec)}{x_i} \frac{1}{\gamma_i}, \yesnumber \label{Eq_grad_xtil}
\end{align*}
where the term $\p{x_i}{\xtil_i} = \gamma_i^{-1}$ is contained along the diagonal of $\Pinv$. The modified gradient-enhanced kernel matrix can thus be calculated with
\begin{align*} 
\Kgtil(\X; \gammavec) 
	&= \Pinv \Kg \Pinv \\
	&=
\begin{bmatrix}
	\K 					& \p{\K}{\ytil_1} 				& \ldots		& \p{\K}{\ytil_d} \\
	\p{\K}{\xtil_1} 	& \pp{\K}{\xtil_1}{\ytil_1} 	& \ldots 	& \pp{\K}{\xtil_1}{\ytil_d} \\
	\vdots 				& 	\vdots							& \ddots 	& \vdots \\
	\p{\K}{\xtil_d} 	& \pp{\K}{\xtil_d}{\ytil_1} 	& \ldots 	& \pp{\K}{\xtil_d}{\ytil_d}
\end{bmatrix} \\
&= 
	\begin{bmatrix}
				\K & \Rtil_1 \odot \K & \ldots & \Rtil_d \odot \K \\ 
	-\Rtil_1 \odot \K & \left( \onebar - \Rtil_1 \odot \Rtil_1\right) \odot \K & \ldots & -\Rtil_1 \odot \Rtil_d \odot \K \\
	\vdots & \vdots & \ddots & \vdots	\\
	-\Rtil_d \odot \K & - \Rtil_1 \odot \Rtil_d \odot \K & \ldots & \left( \onebar - \Rtil_d \odot \Rtil_d\right) \odot \K \\
		\end{bmatrix}, \yesnumber \label{Eq_Kgtil}
\end{align*}
where $\onebar = \one_{n_x \times n_x}$, $\Rtil_i = \gamma_i \R_i(\X) = \R_i(\tilde{X})$, and $\tilde{X} = \X \Pmat$. The modified gradient-enhanced covariance matrix is
\begin{equation}
	\Sigmagtil(\X; \gammavec, \etaKgtil) = \sigK^2 \left( \Kgtil(\X; \gammavec) + \etaKgtil \Idty \right).
\end{equation}

%--------------------------------------------------------------------------
\begin{figure}[t]
	\centering
	\begin{subfigure}[t]{0.49\textwidth}
	\includegraphics[width=\textwidth]{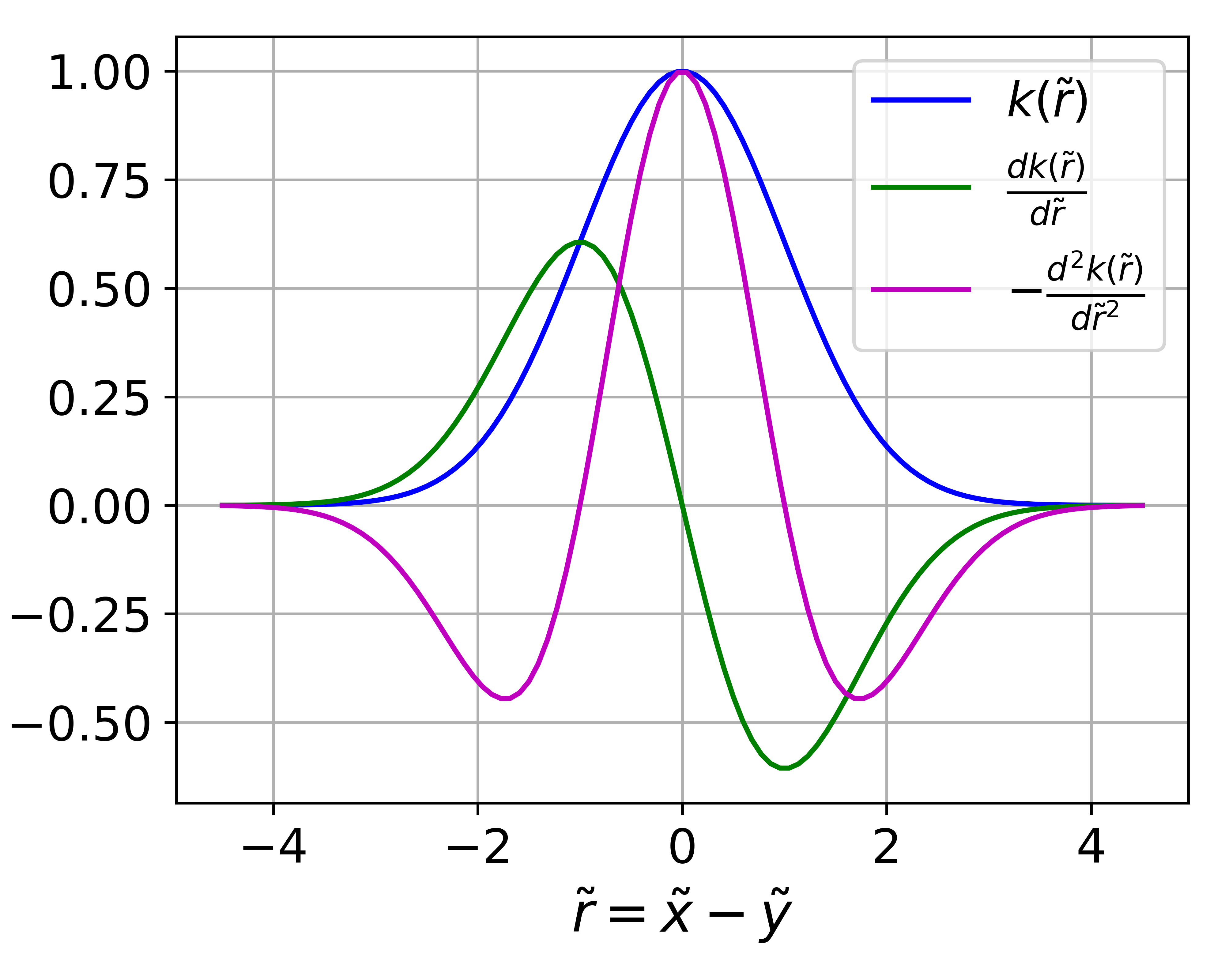}
	\caption{Correlations in one dimension}
	\label{Fig_CorrGaussianKernel_1d}
	\end{subfigure}
	\hspace{0.05cm}
	\begin{subfigure}[t]{0.49\textwidth}
	\includegraphics[width=\textwidth]{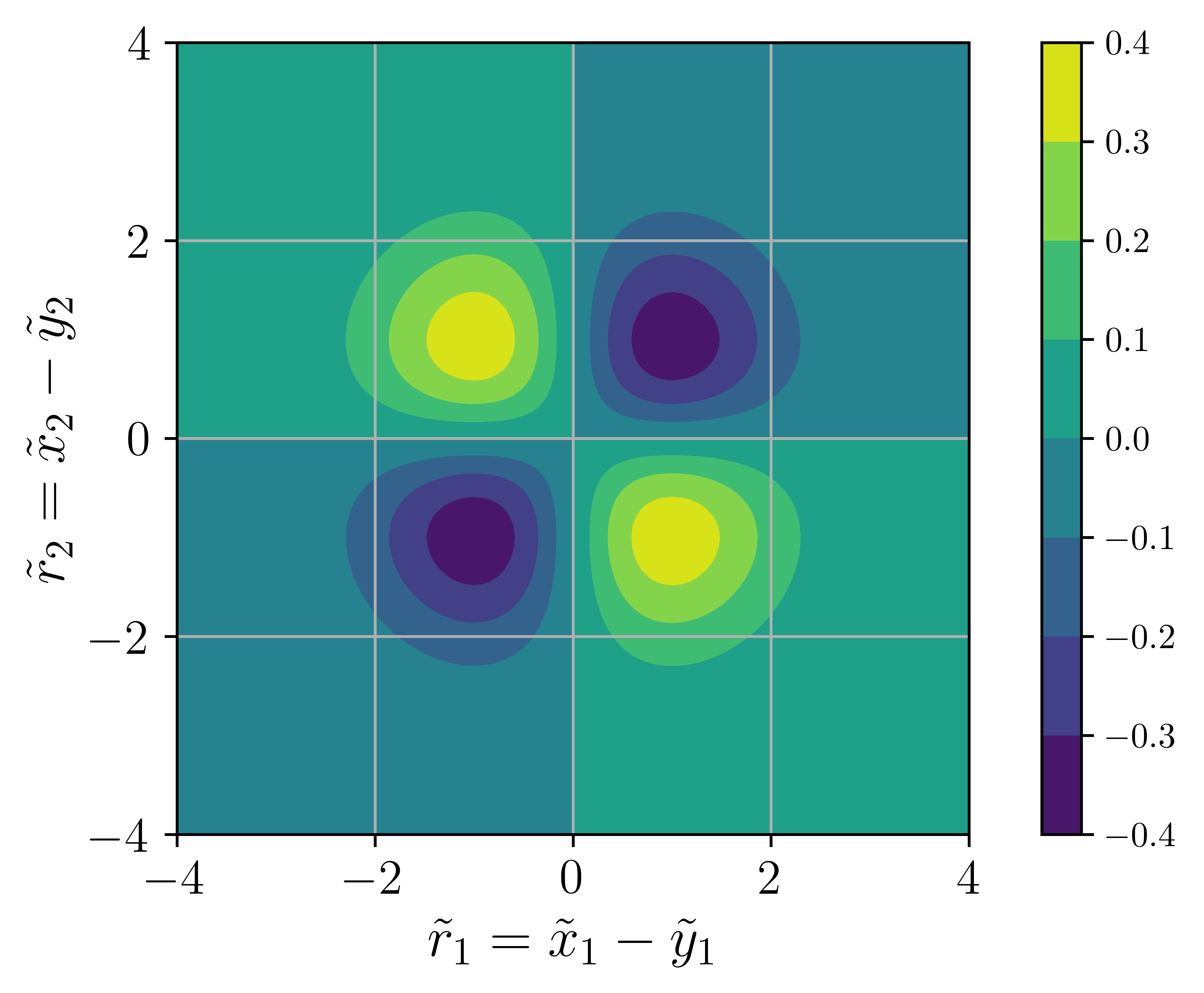}
	\caption{Correlation between $\p{f(\xvec)}{\xtil_1}$ and $\p{f(\yvec)}{\xtil_2}$}
	\label{Fig_CorrGaussianKernel_2d}
	\end{subfigure}
	\caption{Correlations for the Gaussian kernel from \Eq{Eq_Gaussian_kernel} for the function evaluations of a generic function $f(\xvec)$ and its gradient. The term $k(\rtil)$ is the correlation between $f(\xvec)$ and $f(\yvec)$, $\p{k(\rtil)}{\rtil}$ is the correlation between $f(\xvec)$ and $\p{f(\yvec)}{x}$, and the correlation between $\p{f(\xvec)}{x_i}$ and $\p{f(\yvec)}{x_j}$ is $-\pp{k(\rtil)}{\rtil_i}{\rtil_j}$.}
	\label{Fig_CorrGaussianKernel}
\end{figure}
%--------------------------------------------------------------------------

The correlations for the entries in $\fgrad$ can be seen in \Fig{Fig_CorrGaussianKernel}. Having $\Kgtil$ as a correlation matrix makes the GP easier to interpret. Values in $\Kgtil$ close to $-1$ or $1$ indicate near perfect inverse or direct correlation, respectively. On the other hand, values close to $-1$ and $1$ in $\Kg$ indicate negative and positive relations, respectively, but provide little insight on the strength of the relations between the evaluation points.

%--------------------
% New section
%--------------------
\section{Bounding the condition number of the covariance matrix} \label{Sec_WellCondMtd}

%--------------------
% New section
%--------------------
\subsection{The use of a nugget} \label{Sec_WellCondMtd_Nugget}

A common approach to alleviate the ill-conditioning of a matrix is to add a nugget to its diagonal. For a GP, the addition of a nugget to the covariance matrix is analogous to having noisy data \cite{rasmussen_gaussian_2006,ameli_noise_2022}. When the nugget is zero, the surrogate from the GP will match the function of interest exactly at all points where it has been evaluated. The same applies to the evaluated gradients if a gradient-enhanced covariance matrix is used. However, if a positive nugget is used, the surrogate will generally not match the function of interest exactly at points where it has been sampled. It is therefore desirable to use the smallest nugget value required to ensure the condition number of the covariance matrix is below a desired threshold.

The eigenvalues of $\K$, $\Kg$, and $\Kgtil$ are real since they are symmetric matrices. We derive the required nugget to have $\kappa(\K + \etaK \Idty) \leq \condmax$ when the condition number is based on the $\ell_2$ norm:
\begin{align*}
	\kappa(\K + \etaK \Idty) 
		&= \frac{\eigmax + \etaK}{\eigmin + \etaK} \leq \frac{\eigmax}{\etaK} + 1 \leq \condmax \\
	\etaK 
		&\geq \frac{\eigmax}{\condmax - 1}, \yesnumber \label{Eq_eta_wrt_eigmax}
\end{align*}
where $\eigmin$ and $\eigmax$ are the smallest and largest eigenvalues of $\K$, respectively. The results are analogous for $\Kg$ and $\Kgtil$. For positive semidefinite matrices, such as $\K$, $\Kg$, and $\Kgtil$, we have $\eigmax(\K) \leq \trace{\K}$. From \Eq{Eq_eta_wrt_eigmax} it thus follows that sufficient nugget values to bound the condition numbers of the kernel matrices below $\condmax$ are
\begin{align}
	\etaK 		
		&= \frac{n_x}{\condmax - 1} \label{Eq_etaK}\\
	\etaKg(\gammavec) 		
		&= \frac{n_x(\one^\top \gammavec^2 + 1)}{\condmax - 1} \label{Eq_etaKg_w_trace} \\ 
	\etaKgtil 		
		&= \frac{n_x(d + 1)}{\condmax - 1}. \label{Eq_req_etaKgtil_w_trace}
\end{align}
These are sufficient but not necessary conditions to ensure that the condition numbers of the covariance matrices are smaller than $\condmax$ since the bound ${\eigmax < \trace{\K}}$ is not tight and \Eq{Eq_eta_wrt_eigmax} was derived with the worst case ${\eigmin = 0}$.

\Eq{Eq_etaK} provides a sufficiently small $\etaK$ to ensure that ${\kappa(\Sigma) \leq \condmax}$. However, $\etaKg$ from \Eq{Eq_etaKg_w_trace} is undesirable since it depends on $\gammavec$. Consequently, as $\gammavec$ gets larger, $\etaKg$ must also get bigger to ensure that $\kappa(\Sigmag(\gammavec)) \leq \condmax$. The method from \cite{marchildon_non-intrusive_2023} to address this is summarized in \Sec{Sec_WellCondMtd_Rescale}. Finally, $\etaKgtil$ from \Eq{Eq_req_etaKgtil_w_trace} ensures that $\kappa(\Sigmagtil) \leq \condmax$ with no dependence on $\gammavec$ and could be used on its own. However, a tighter bound on $\etaKgtil$ is derived in \Sec{Sec_WellCondMtd_Precon}. that scales with $\mathcal{O}(n_x \sqrt{d})$ instead of $\mathcal{O}(n_x d)$ when \Eq{Eq_req_etaKgtil_w_trace} is used. In the remaining subsections three methods are introduced that bound the condition number of the gradient-enhanced covariance matrix.

%--------------------
% New subsection
%--------------------

\subsection{Baseline method: constrained optimization of $\gammavec$} \label{Sec_WellCondMtd_Baseline}

One previous approach that has been used to ensure that $\kappa(\Sigmag) \leq \condmax$ is to add a constraint to the maximization of the marginal log-likelihood from \Eq{Eq_ln_lkd_final} \cite{won_maximum_2006}. The hyperparameters $\gammavec$ are thus selected by solving the following constrained optimization problem:
\begin{equation} \label{Eq_optz_lkd}
	\gammavec^*
		= \argmax_{\gammavec > 0} L(\gammavec) \quad \text{s.t.} \quad \kappa(\Sigmag(\gammavec)) \leq \condmax.
\end{equation}
There will always be a feasible solution to \Eq{Eq_optz_lkd} if $\etaKg \geq \frac{n_x}{\condmax -1}$. This can be verified from \Eq{Eq_etaKg_w_trace} with $\gammavec \rightarrow \zero_d$. Solving \Eq{Eq_optz_lkd} to set the hyperparameters thus ensures that $\kappa(\Sigmag(\X, \gammavec)) \leq \condmax \forall \, \X \in \mathbb{R}^{n_x \times d}$. However, the constraint in \Eq{Eq_optz_lkd} may result in selecting hyperparameters that provide a significantly lower marginal log-likelihood. This impacts the accuracy of the surrogate and is shown in \Sec{Sec_Results_Optz} to be detrimental to the optimization.

%--------------------
% New subsection
%--------------------
\subsection{Rescaling method} \label{Sec_WellCondMtd_Rescale}

A short overview of the rescaling method from \cite{marchildon_non-intrusive_2023} is provided in this section; the proofs can be found in \cite{marchildon_non-intrusive_2023}. To ensure that $\kappa(\Sigmag(\gammavec; \etaKg)) \leq \condmax$ when $\gamma_1 = \ldots = \gamma_d$ and $\Sigmag$ is not diagonally dominant, the parameter space is rescaled such that the minimum Euclidean distance between evaluation points is $\vminset$, where
\begin{align}
	\vminset(d, n_x) 
		&= \min \left( 2 \sqrt{d}, \frac{2 + \sqrt{4 + 2e^2 \ln \left( \frac{(n_x - 1)(1 + 2 \sqrt{d} )}{2} \right) }}{e} \right).
\end{align}
The condition that $\gamma_1 = \ldots = \gamma_d$ can be achieved by rescaling the data non-isotropically \cite{marchildon_non-intrusive_2023}. Meanwhile, the condition that $\Sigmag$ is not diagonally dominant is required since the condition number of $\Sigmag$ is unbounded as $\gammavec$ tends to infinity, regardless of the selected nugget. However, since the condition on the diagonal dominance of $\Sigmag$ only applies for large values of $\gammavec$, it is not in practice limiting since there is little correlation between evaluation points and thus the marginal log-likelihood is unlikely to be maximized at these values of $\gammavec$ \cite{marchildon_non-intrusive_2023}. 

For the initial isotropic rescaling we use
\begin{align*}
	\X 
		&= \tau \X_{\text{initial}} \\
	\p{f(\xvec)}{x_i} 
		&= \frac{1}{\tau} \left(\p{f(\xvec)}{x_i} \right)_{\text{initial}},
\end{align*}
where $\tau = \frac{\vminset}{v_{\text{min,initial}}}$. The required nugget to bound the condition number of $\Sigmag$ with the rescaling method is
\begin{equation}
	\etaKg(d, n_x) 
		= \frac{1 + (n_x -1) \frac{2 \sqrt{d}}{\vreq} e^{\frac{\vreq}{2 \sqrt{d}} -1}}{\condmax - 1}.
\end{equation}
The hyperparameters are set by solving the optimization problem from \Eq{Eq_optz_lkd}. Since the bound $\kappa(\Sigmag(\gammavec; \etaKg)) \leq \condmax$ is only ensured when $\gamma_1 = \ldots = \gamma_d$, the maximization of the marginal log-likelihood needs to contain a constraint on the condition number to ensure that it does not exceed $\condmax$, which could cause the Cholesky decomposition to fail. If the constraint is not active, even if all the hyperparameters are not equal, then no further iterations are required and the optimized hyperparameters can be used. However, if the constraint on the condition number is active, then a non-isotropic rescaling of $\X$ and the gradients can be performed to get an unconstrained solution to \Eq{Eq_optz_lkd}. The details on the non-isotropic rescaling can be found in \cite{marchildon_non-intrusive_2023}. 

%--------------------
% New section
%--------------------

\subsection{Preconditioning method with $\etaKgtil = \mathcal{O}(n_x \sqrt{d})$} \label{Sec_WellCondMtd_Precon}

\Eq{Eq_req_etaKgtil_w_trace} can be used to select $\etaKgtil$ but it results in $\etaKgtil = \mathcal{O}(n_x d)$. In this section a smaller sufficient nugget is derived such that $\etaKgtil = \mathcal{O}(n_x \sqrt{d})$ is sufficient to ensure that $\kappa(\Sigmagtil(\gammavec;\etaKgtil)) \leq \condmax \, \forall \, \gammavec > 0$. The derivation uses the Gershgorin circle theorem, which bounds the largest eigenvalue of a symmetric matrix $\A$ by 
\begin{equation}
	\eigmax(\A) \leq \max_{i} (a_{ii} + \sum_{j\neq i} |a_{ij} |).
\end{equation}
The two following propositions provide an upper bound on the sum of the absolute values of the off-diagonal entries of $\Kgtil$ when it is constructed with the Gaussian kernel from \Eq{Eq_Gaussian_kernel}. 

\begin{proposition} \label{Prop_ura}
	For $n_x, d \in \mathbb{Z}^+$ the sum of the absolute values for the off-diagonal entries for any of the first $n_x$ rows of $\Kgtil$ from \Eq{Eq_Kgtil} using the Gaussian kernel is bounded by $u_{\text{G}}$, where
\begin{equation} \label{Eq_ub_Kgtil_G}
	u_{\text{G}}(n_x,d) = (n_x-1) \frac{1 + \sqrt{1 + 4d}}{2} e^{-\frac{1 + 2d - \sqrt{1 + 4d}}{4d}}.
\end{equation}	
	
\end{proposition}

\begin{proof}
{
	We derive an upper bound for the sum of the absolute values for the off-diagonal entries for any of the first $n_x$ rows of $\Kgtil$ from \Eq{Eq_Kgtil}. The derivation is the same for any of the rows and we thus consider the $a$-th row, where $1 \leq a \leq n_x$ and
\begin{align*}
	\sum_{\substack{i=1 \\ i\neq a}}^{n_x} \left( \Kgtil \right)_{ai} 
		&= \sum_{\substack{i=1 \\ i\neq a}}^{n_x} \left( 1 + \sum_{j=1}^d \left| \xtil_{aj} - \xtil_{ij} \right| \right) e^{-\frac{\| \xtil_{a:} - \xtil_{i:} \|_2^2 }{2}} \\
		& \leq (n_x - 1) \max_{i} \left( 1 + \sum_{j=1}^d \left| \xtil_{aj} - \xtil_{ij} \right| \right) e^{-\frac{\| \xtil_{a:} - \xtil_{i:} \|_2^2 }{2}} \\
		&\leq (n_x -1) \max_{\wvec \geq 0} \left( \left(1 + \wvec^\top \one_d \right) e^{-\frac{\wvec^\top \wvec}{2}} \right),
\end{align*} 
where $w_j = |\xtil_{aj} - \xtil_{ij}|$ and we denote the expression inside the max function as $g(\wvec)$. To identify the maximum of $g(\wvec)$ we calculate its derivative and set it to zero:
\begin{align*}
	\p{g(\wvec)}{w_i} 
		&= \left( 1 - w_i (1 + \wvec^\top \one_d) \right) e^{-\frac{\wvec^\top \wvec}{2}} = 0 \\
	w_i &= \frac{1}{1 + \wvec^\top \one} \, \forall \, i \in \{1, \ldots, d \}.
\end{align*}
It is clear that the gradient of $g(\wvec)$ is zero if and only if all of the entries in $\wvec$ are equal. We thus use $\wvec = \alpha \one_d$ and solve for the value of $\alpha$ that maximizes $g(\alpha \one_d)$:
\begin{align*}
	\p{g(\alpha \one_d)}{\alpha} 
		= d\left(1 - \alpha (1 + d \alpha) \right) e^{-\frac{d \alpha^2}{2}} 
		&= 0 \\
%	d \alpha^2 + \alpha -1 
%		&= 0 \\
	\alpha^* &= \frac{-1+ \sqrt{1 + 4d}}{2d},
\end{align*}
where we only kept the positive root since $\wvec \geq 0$ and it is straightforward to verify that this provides the maximum of $g(\wvec)$. \Eq{Eq_ub_Kgtil_G} is recovered by evaluating $g(\alpha^* \one_d)$, which completes the proof. 
}
\end{proof}

\begin{proposition} \label{Prop_urb}
	The sum of the absolute values for the off-diagonal entries for any of the last $n_x d$ rows of $\Kgtil$ using the Gaussian kernel is smaller than $u_{\text{G}}(n_x,d)$ from \Eq{Eq_ub_Kgtil_G} for $n_x, d \in \mathbb{Z}^+$.	
\end{proposition}

\begin{proof}
	The proof can be found in \Sec{Sec_ApxProof_Prop_urb}.
\end{proof}

As a result of \Props{Prop_ura}{Prop_urb} and the Gershgorin circle theorem we have $\eigmax(\Kgtil) \leq 1 + u_{\text{G}}(n_x,d)$, where $u_{\text{G}}(n_x,d)$ is from \Eq{Eq_ub_Kgtil_G}. Therefore, we can ensure that $\kappa(\Kgtil(\gammavec) + \etaKgtil \Idty) \leq \condmax \, \forall \, \gammavec >0$ with \Eqs{Eq_eta_wrt_eigmax}{Eq_ub_Kgtil_G}:
\begin{equation} \label{Eq_etaKgtil}
	\etaKgtil(n_x, d) = \frac{1 + (n_x-1) \frac{1 + \sqrt{1 + 4d}}{2} e^{-\frac{1 + 2d - \sqrt{1 + 4d}}{4d}}}{\condmax - 1}.
\end{equation}
The following lemma proves that \Eq{Eq_etaKgtil} provides $\etaKgtil = \mathcal{O}(n_x \sqrt{d})$.

\begin{lemma} \label{Lem_trend_etaKgtil}
	From \Eq{Eq_etaKgtil} we have $\etaKgtil(n_x, d) = \mathcal{O}(n_x \sqrt{d})$ for $n_x, d \in \mathcal{Z}^+$ since 
\begin{equation}
		\etaKgtil(n_x, d) < \frac{1 + (n_x - 1)(1 + \sqrt{d}) e^{-\frac{3 - \sqrt{5}}{4} }}{\condmax -1}.
\end{equation}
\end{lemma}

\begin{proof}
	The proof can be found in \Sec{Sec_ApxProof_Lem_trend_etaKgtil}. 
\end{proof}

From \Lem{Lem_trend_etaKgtil} it is clear that the use of \Eq{Eq_etaKgtil} to calculate $\etaKgtil$ is advantageous for high dimensional problems since it provides $\etaKgtil = \mathcal{O}(n_x \sqrt{d})$ instead of $\etaKgtil = \mathcal{O}(n_x d)$ from \Eq{Eq_req_etaKgtil_w_trace}.

%--------------------
% New section
%--------------------
\section{Implementation} \label{Sec_Implementation}

Since the inverse of $\Kg + \etaKg \Idty$ is needed to calculate $\mu_f$, $\sigma_f^2$, and $\ln(L)$ along with their gradients, it is desirable to calculate its Cholesky decomposition. Once the Cholesky decomposition has been calculated, it becomes inexpensive to evaluate $\mu_f(\xvec)$ and $\sigma_f^2(\xvec)$ for various $\xvec$. However, doing so directly may cause the decomposition to fail since the condition number of $\Kg + \etaKg \Idty$ cannot be bounded with a finite $\etaKg$, as explained in \Sec{Sec_WellCondMtd_Nugget}. Instead, $\Kg$ can be used to construct $\Kgtil$ and the Cholesky decomposition of $\Kgtil + \etaKgtil \Idty$ is calculated instead. Calculating the Cholesky decomposition of $\Kgtil + \etaKgtil \Idty$ is numerically stable since $\kappa(\Kgtil(\gammavec) + \etaKgtil \Idty) \leq \condmax \, \forall \gammavec > 0$, as proven in \Sec{Sec_WellCondMtd}. The following algorithm details how the Cholesky decomposition for the gradient-enhanced covariance matrix should be calculated.

\begin{algorithm}
	\caption{Stable Cholesky decomposition for gradient-enhanced GP} \label{Alg_Stable_Cho}
	\begin{algorithmic}[1]
		\State{Select evaluation points $\X$ and hyperparameters $\gammavec$}
		\State{Calculate $\Kg$, $\Pmat$, and $\etaKgtil$ with \Eqss{Eq_Kg}{Eq_Pmat}{Eq_etaKgtil}, respectively}
		\State{$\Kgtil = \Pinv \Kg \Pinv$}
		\State{$\Ltil \Ltil^\top = \Kgtil + \etaKgtil \Idty$}
		\State{$\Lmat = \Pmat \Ltil$}
	\end{algorithmic}
\end{algorithm}

Algebraically we have the following relation:
\begin{align*}
	\Lmat \Lmat^T 
		= \left( \Kg + \etaKgtil \Pmat \Pmat \right)^{-1}
		= \Pinv \left( \Kgtil + \etaKgtil \Idty \right)^{-1} \Pinv.
\end{align*}
This indicates that the addition of the nugget $\etaKgtil$ after the preconditioning is equivalent to adding a nugget that varies along the diagonal and scales with the squared hyperparameter $\gammavec$. However, the condition number of $\left( \Kgtil + \etaKgtil \Idty \right)$ is bounded below $\condmax$ for $\gammavec > 0$ and can be several orders of magnitude smaller than the condition number for $\left( \Kg + \etaKgtil \Pmat \Pmat \right)$. As such, the Cholesky decomposition of the former should be performed, as detailed in Algorithm \ref{Alg_Stable_Cho}.

%--------------------
% New section
%--------------------
\section{Results} \label{Sec_Results}

%--------------------
% New subsection
%--------------------
\subsection{Condition numbers of the Gaussian kernel covariance matrices} \label{Sec_Results_CondKernSqExp}

The condition numbers of the gradient-free and gradient-enhanced covariance matrices are compared for the baseline method presented in \Sec{Sec_WellCondMtd_Baseline}, the rescaling method from [15], and the preconditioning method introduced in this paper. The covariance matrices depend only on the evaluation points in $\X$, the nugget, and the hyperparameters $\gammavec$. However, the maximization of the marginal log-likelihood also depends on the function of interest and we use the Rosenbrock function:
\begin{equation} \label{Eq_Rosenbrock}
	f(\xvec) = \sum_{i=1}^{d-1} \left[ 10 \left(x_{i+1} - x_i^2 \right)^2 + \left( 1 - x_i \right)^2 \right].
\end{equation}
A Latin hypercube sampling centred around $\xvec = [1, \ldots, 1]^\top$, which is the minimum for the Rosenbrock function, is used to select the evaluation points
\begin{equation} \label{Eq_X_dataset}
	\X = 10^{-3} \times 
\begin{bmatrix}
	\phantom{-}1 & \phantom{-}9 & \phantom{-}7 & -9 & -5 & -7 & -3 & \phantom{-}5 & \phantom{-}3 & -1 \\
	\phantom{-}1 & -3 & \phantom{-}7 & \phantom{-}3 & \phantom{-}5 & -9 & -7 & \phantom{-}9 & -1 & -5
\end{bmatrix}^\top
 + 1,
\end{equation}
where $\vmin = \sqrt{2} / 500 \approx 2.8 \times 10^{-3}$, which is the minimum Euclidean distance between evaluation points.

%--------------------------------------------------------------------------
\begin{figure}[t!]
	\centering
	\begin{subfigure}[t]{0.49\textwidth}
	\includegraphics[width=\textwidth]{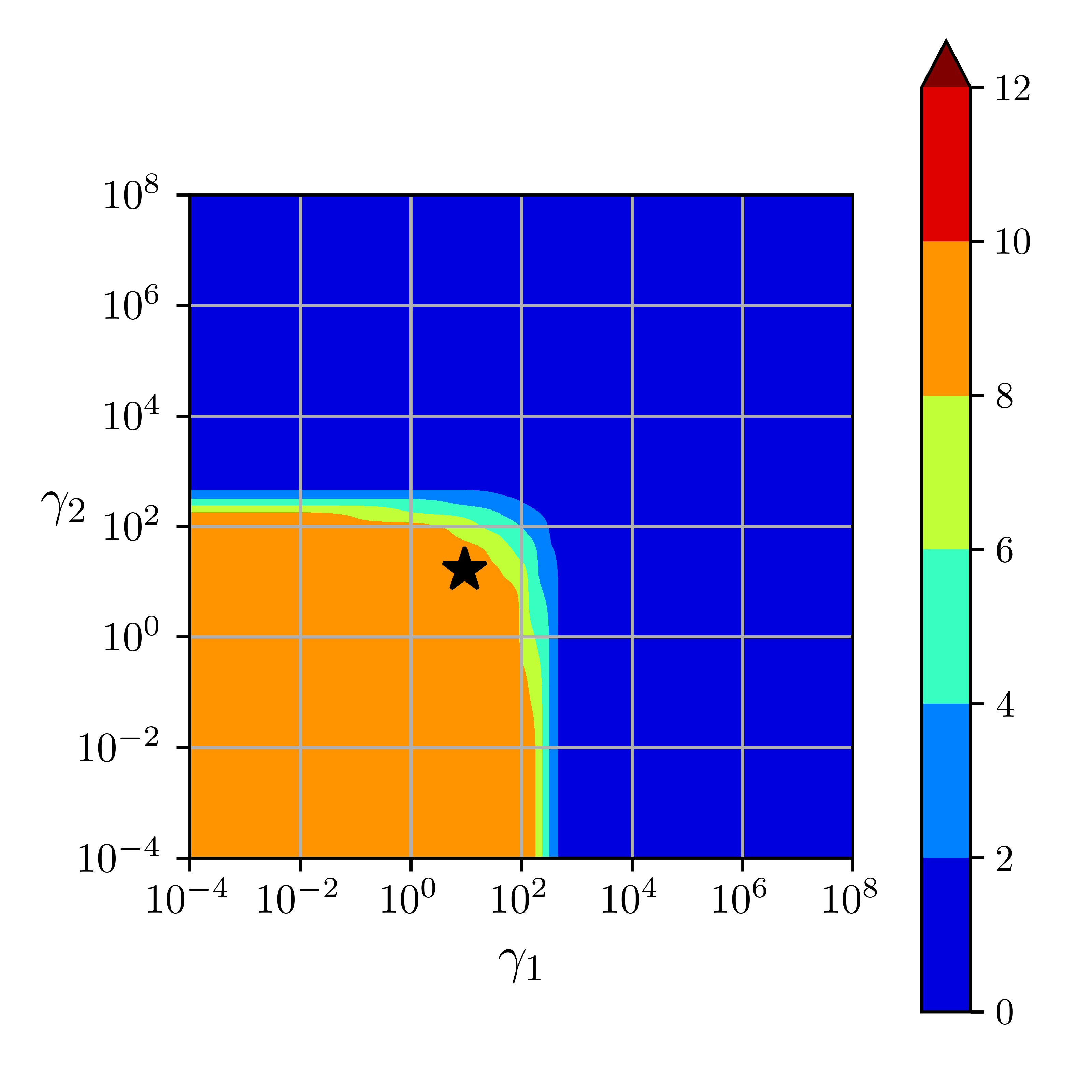}
	\vspace{-0.9cm}
	\caption{Baseline gradient-free: $\log(\kappa(\Sigma))$}
	\label{Fig_cond_2d_SqExp_Kbase}
	\end{subfigure}
	\hspace{0.05cm}
	\begin{subfigure}[t]{0.49\textwidth}
	\includegraphics[width=\textwidth]{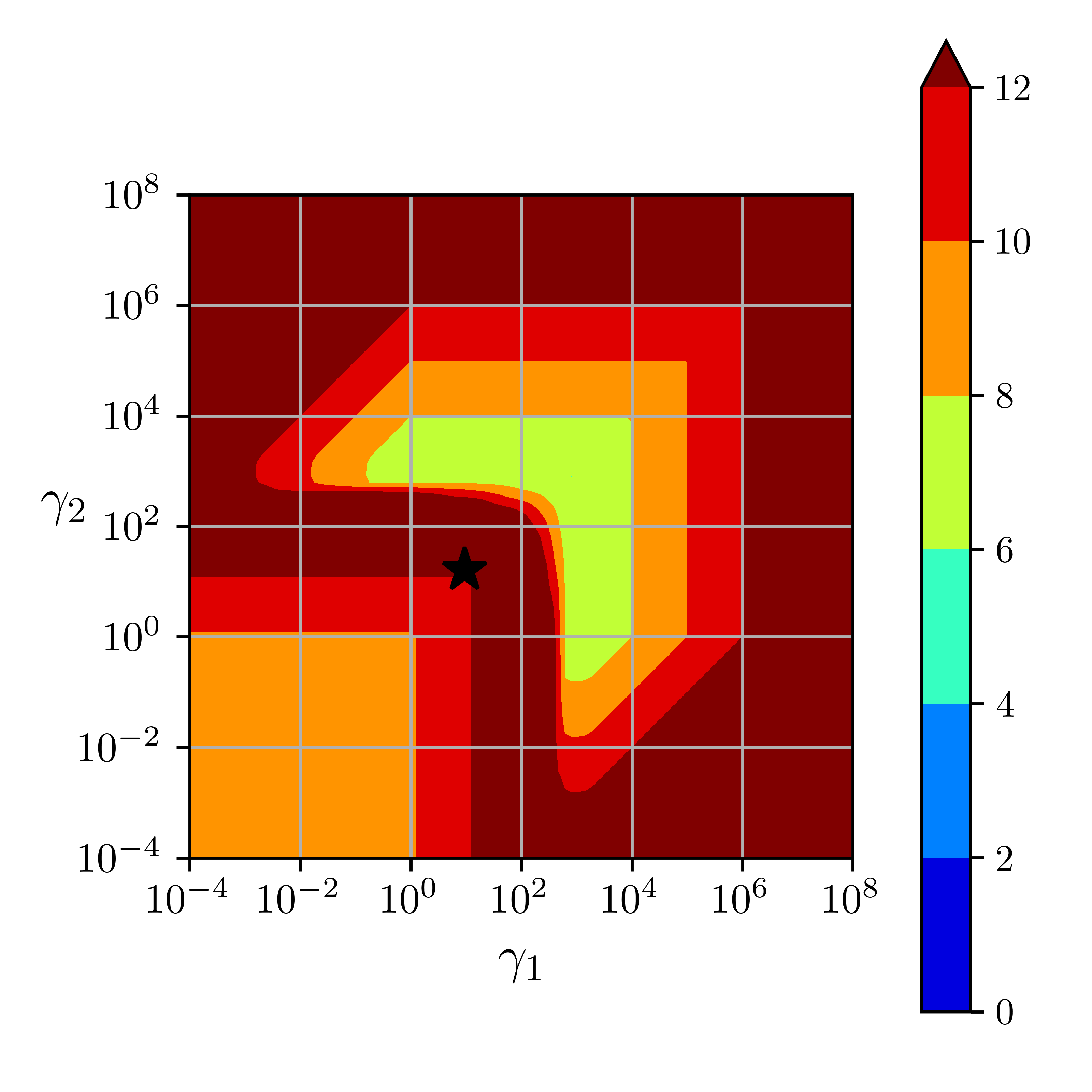}
	\vspace{-0.9cm}
	\caption{Baseline gradient-enhanced: $\log(\kappa(\Sigmag))$}
	\label{Fig_cond_2d_SqExp_Kgrad_baseline}
	\end{subfigure}
	\hspace{0.5cm}
	\begin{subfigure}[t]{0.49\textwidth}
	\includegraphics[width=\textwidth]{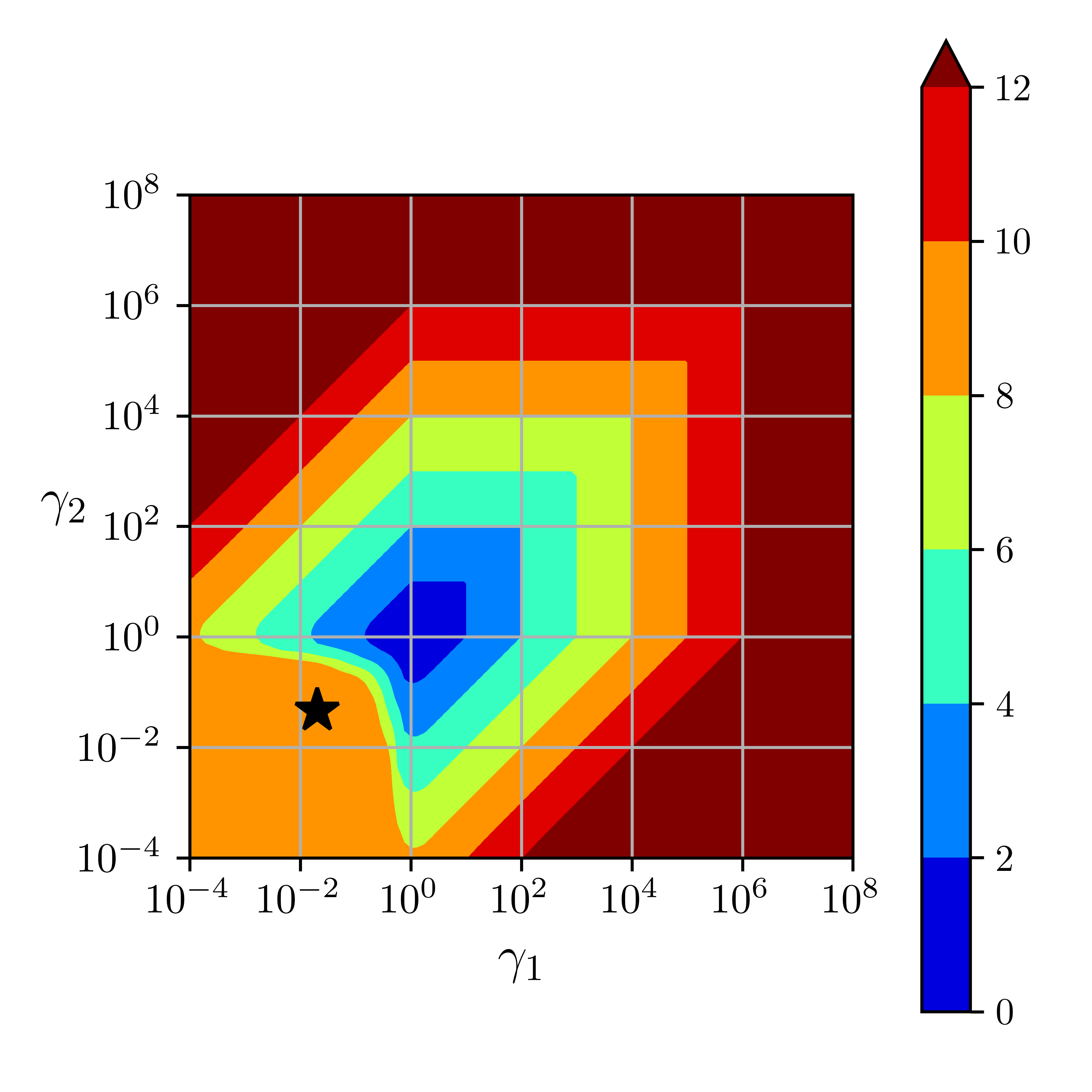}
	\vspace{-0.9cm}
	\caption{Gradient-enhanced with the rescaling method: $\log(\kappa(\Sigmag))$}
	\label{Fig_cond_2d_SqExp_Rescale}
	\end{subfigure}
	\hspace{0.05cm}
	\begin{subfigure}[t]{0.49\textwidth}
	\includegraphics[width=\textwidth]{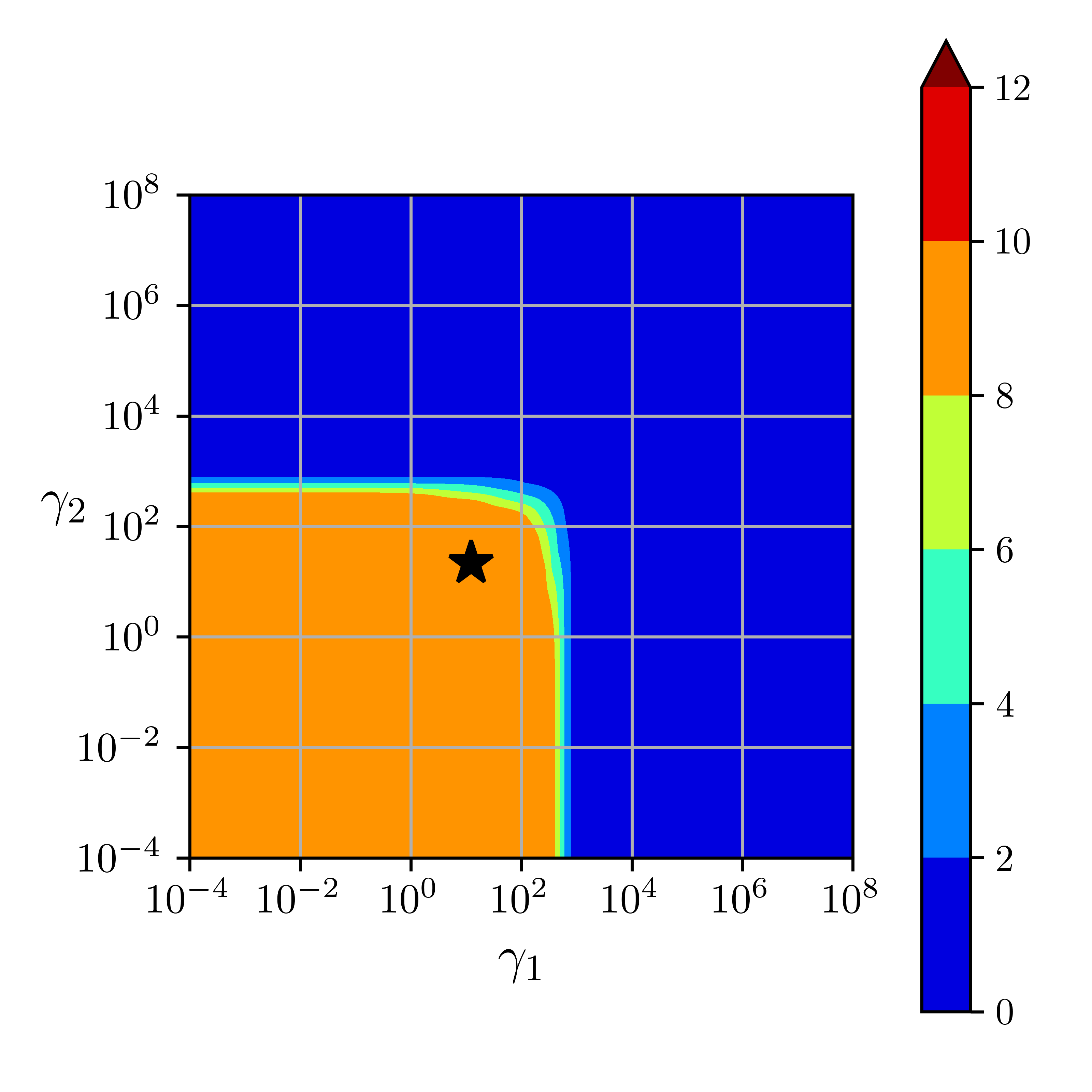}
	\vspace{-0.9cm}
	\caption{Gradient-enhanced with the preconditioning method: $\log(\kappa(\Sigmagtil))$}
	\label{Fig_cond_2d_SqExp_Precon}
	\end{subfigure}
	\caption{The condition number for covariance matrices using the nugget $\etaKgtil = 1.5 \times 10^{-9}$ from \Eq{Eq_etaKgtil} along with the set of evaluation points $\X$ from \Eq{Eq_X_dataset}, which has a minimum Euclidean distance between evaluation points of $\vmin = 2.8 \times 10^{-3}$. The star marking indicates where the marginal log-likelihood function from \Eq{Eq_ln_lkd_final} is maximized with the use of the Rosenbrock function from \Eq{Eq_Rosenbrock}.}
	\label{Fig_cond_2d_SqExp}
\end{figure}
%--------------------------------------------------------------------------

\Fig{Fig_cond_2d_SqExp} plots the condition number of the covariance matrices as a function of $\gammavec$ for the evaluation points from \Eq{Eq_X_dataset}. The star marker indicates where the marginal log-likelihood from \Eq{Eq_ln_lkd_final} is maximized. The nugget value for all cases is $\eta = 1.5 \times 10^{-9}$, which comes from \Eq{Eq_etaKgtil}. Red regions in \Fig{Fig_cond_2d_SqExp} indicate where the condition number is greater than $\condmax = 10^{10}$. 

\Figs{Fig_cond_2d_SqExp_Kbase}{Fig_cond_2d_SqExp_Kgrad_baseline} plot the condition number of the gradient-free and gradient-enhanced covariance matrices, respectively, using the baseline method, which does not precondition the covariance matrix but adds the nugget to its diagonal. For the gradient-free case we have $\kappa(\Sigma(\gammavec)) < \condmax \, \forall \, \gammavec > 0$. However, for the gradient-enhanced case $\kappa(\Sigmag(\gammavec)) \geq \condmax$ for most values of $\gammavec$, including where the marginal log-likelihood from \Eq{Eq_ln_lkd_final} is maximized. Selecting the hyperparameters to satisfy the constraint $\kappa(\Sigmag(\gammavec)) \leq \condmax$ results in a lower marginal log-likelihood. This impacts the accuracy of the surrogate, which degrades the performance of the Bayesian optimizer, as will be shown in \Sec{Sec_Results_Optz}.

\Fig{Fig_cond_2d_SqExp_Rescale} plots $\log(\kappa(\Sigmag))$ with the use of the rescaling method. As stated in \Sec{Sec_WellCondMtd_Rescale}, the rescaling method only ensures that $\kappa(\Sigmag(\gammavec)) \leq \condmax$ when $\gamma_1 = \ldots = \gamma_d$ and $\Sigmag$ is not diagonally dominant. While there are several values of $\gammavec$ where $\kappa(\Sigmag(\gammavec)) \geq \condmax$, the condition number is below $\condmax$ where the marginal log-likelihood is maximized.

From \Fig{Fig_cond_2d_SqExp_Precon} we have $\kappa(\Kgtil) < \condmax \, \forall \, \gammavec > 0$ as a result of the preconditioning method. Consequently, there is no need for a constraint on the condition number for the optimization of the hyperparameters. This ensures that the hyperparameters are never constrained by the need to bound the condition number of the covariance matrix. Furthermore, this also provides a small reduction in the cost of the hyperparameter optimization since the constraint and its gradient do not need to be calculated. 

The blue in \Figs{Fig_cond_2d_SqExp_Kbase}{Fig_cond_2d_SqExp_Precon} indicate regions where $\Sigma$ and $\Sigmagtil$ are nearly equal to the identity matrix. While the condition number is very small in these regions, the marginal log-likelihood is as well. This makes it undesirable to select those values of $\gammavec$.

%--------------------
% New subsection
%--------------------
\subsection{Applications to other kernels} \label{Sec_Results_CondMiscKern}

The preconditioning method can be applied to gradient-enhanced covariance matrices that utilize kernels other than the Gaussian kernel considered thus far. For example, the preconditioning method can be applied to the Mat\'ern $\frac{5}{2}$ and rational quadratic kernels \cite{rasmussen_gaussian_2006}:
\begin{align}
	k_{\text{M}\frac{5}{2}}(\rvectil) 	
		&= \left(1 + \sqrt{3} \| \rvectil \| + \| \rvectil \|^2 \right) e^{- \sqrt{3} \| \rvectil \|} \label{Eq_kern_Mat5f2}\\
	k_{rq}(\rvectil) 
		&= \left(1 + \frac{ \| \rvectil \|^2 }{2 \alpha} \right)^{-\alpha}, \label{Eq_kern_RatQd}
\end{align}
where $\alpha > 0$ is a hyperparameter and $\rtil = \gamma_i (x_i - y_i)$. The hyperparameters for the Mat\'ern $\frac{5}{2}$ and rational quadratic kernels from \Eqs{Eq_kern_Mat5f2}{Eq_kern_RatQd}, respectively, have been selected such that the preconditioning matrix required to make their gradient-enhanced kernel matrix a correlation matrix also comes from \Eq{Eq_Pmat_gen}. Using \Eq{Eq_etaKg_w_trace} provides $\etaKgtil = \mathcal{O}(n_x d)$ and ensures that $\bSigmagtil$. Alternatively, \Eq{Eq_etaKgtil} could be used to have $\etaKgtil = \mathcal{O}(n_x \sqrt{d})$, but this does not provide a provable upper bound for $\kappa(\Sigmagtil)$ since \Eq{Eq_etaKgtil} is specific to the Gaussian kernel. The methodology used in \Sec{Sec_WellCondMtd_Precon} to derive a nugget $\etaKgtil = \mathcal{O}(n_x \sqrt{d})$ that ensures $\bSigmagtil$ could also be applied to other kernels, such as the Mat\'ern $\frac{5}{2}$ and rational quadratic kernels.

\Fig{Fig_cond_2d_Kern} plots the condition number of the baseline and preconditioned gradient-enhanced covariance matrices for the Mat\'ern $\frac{5}{2}$ and rational quadratic kernels. The set of evaluation points in $\X$ and the nugget come from \Eqs{Eq_X_dataset}{Eq_etaKgtil}, respectively, which are the same as the ones used for \Fig{Fig_cond_2d_SqExp}. It is clear from \Figs{Fig_cond_2d_Kern_Mat5f2_Kbase}{Fig_cond_2d_Kern_RatQd_Kbase} that the condition number for the baseline method, \ie the non-preconditioned gradient-enhanced covariance matrices, for both kernels is larger than $\condmax$ for several values of $\gammavec$, including where the marginal log-likelihood is maximized at the star marker. However, with the preconditioning method we have $\bSigmagtil$ for both kernels as seen in \Figs{Fig_cond_2d_Kern_Mat5f2_precon}{Fig_cond_2d_Kern_RatQd_precon}. This demonstrates that the gradient-enhanced covariance matrix constructed with various kernels suffers from severe ill-conditioning. Fortunately, the preconditioning method can be applied to bound the condition number of $\Sigmagtil$ constructed with various kernels. The provable bound $\bSigmagtil$ requires the nugget to be calculated from \Eq{Eq_etaKg_w_trace}, but using \Eq{Eq_etaKgtil} was sufficient for the case considered in \Fig{Fig_cond_2d_Kern}. A user of a gradient-enhanced GP with a non-Gaussian kernel could start by using a nugget calculated with \Eq{Eq_etaKgtil}, and switch to using \Eq{Eq_etaKg_w_trace} instead if the condition number of $\Sigmagtil$ was found to exceed $\condmax$.

%--------------------------------------------------------------------------
\begin{figure}[t!]
	\centering
	\begin{subfigure}[t]{0.49\textwidth}
	\includegraphics[width=\textwidth]{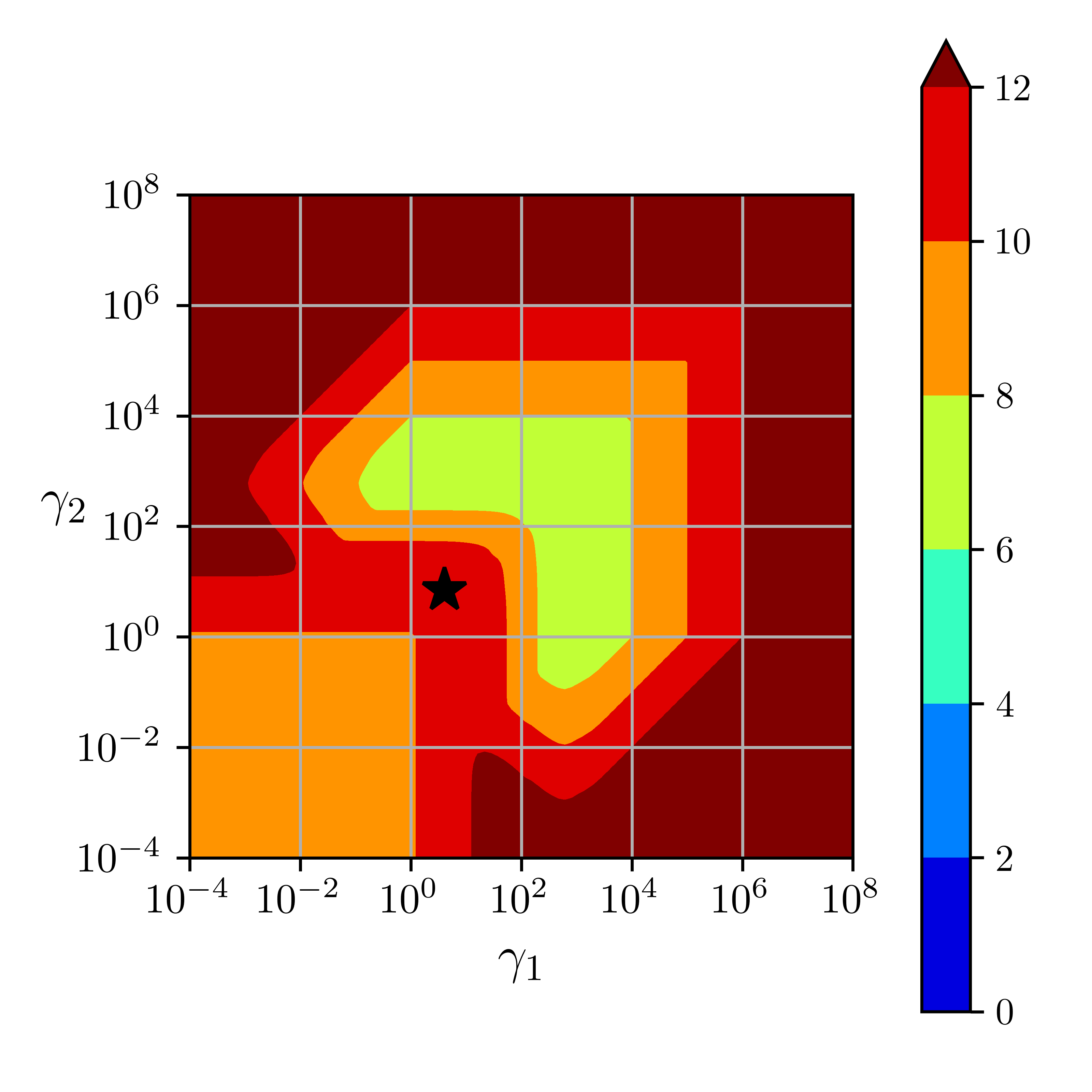}
	\vspace{-0.9cm}
	\caption{Baseline Mat\'ern $\frac{5}{2}$ kernel: $\log(\kappa(\Sigmag))$}
	\label{Fig_cond_2d_Kern_Mat5f2_Kbase}
\end{subfigure}
\hspace{0.05cm}
\begin{subfigure}[t]{0.49\textwidth}
	\includegraphics[width=\textwidth]{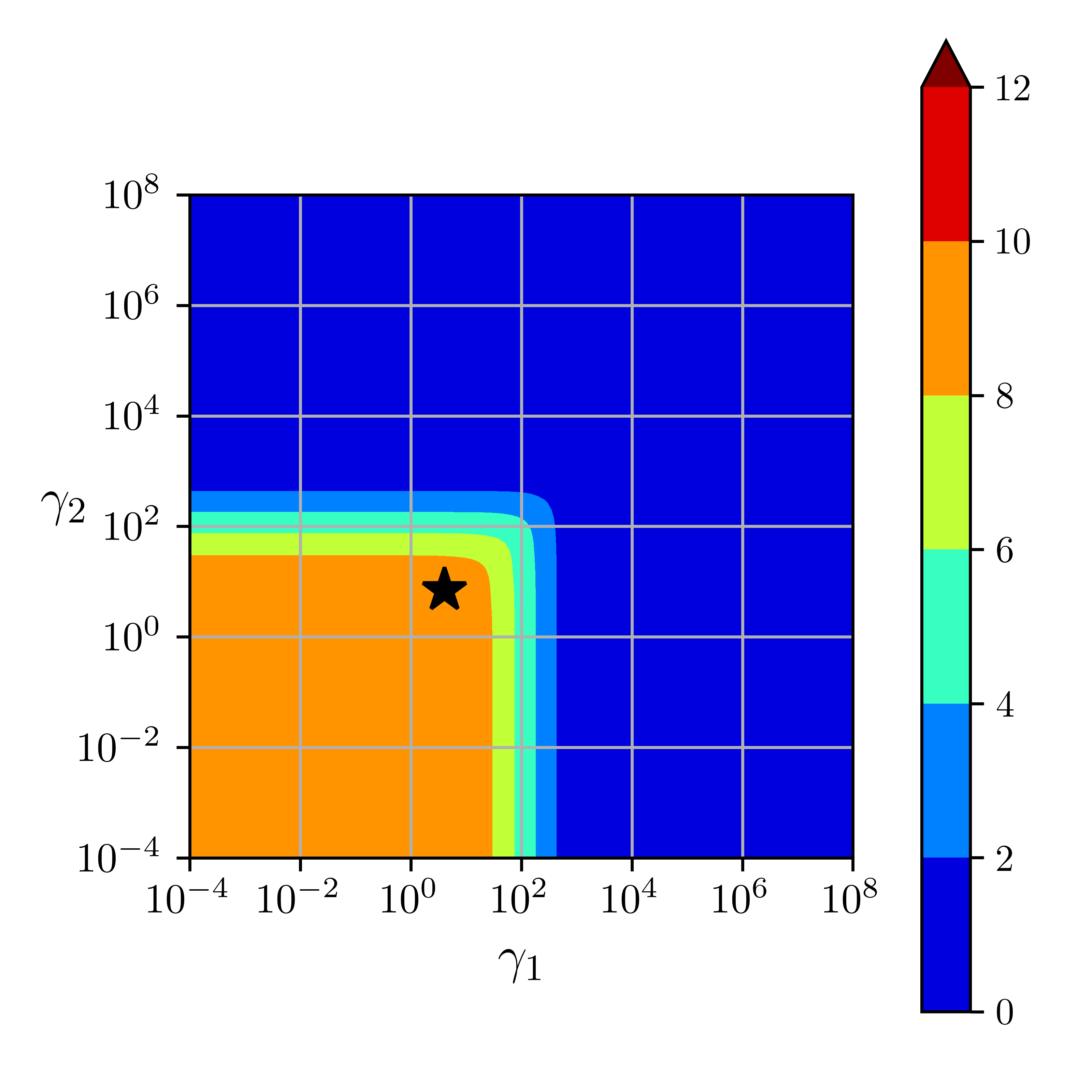}
	\vspace{-0.9cm}
	\caption{Preconditioned Mat\'ern $\frac{5}{2}$ kernel: $\log(\kappa(\Sigmagtil))$}
	\label{Fig_cond_2d_Kern_Mat5f2_precon}
\end{subfigure}
\hspace{0.5cm}
\begin{subfigure}[t]{0.49\textwidth}
	\includegraphics[width=\textwidth]{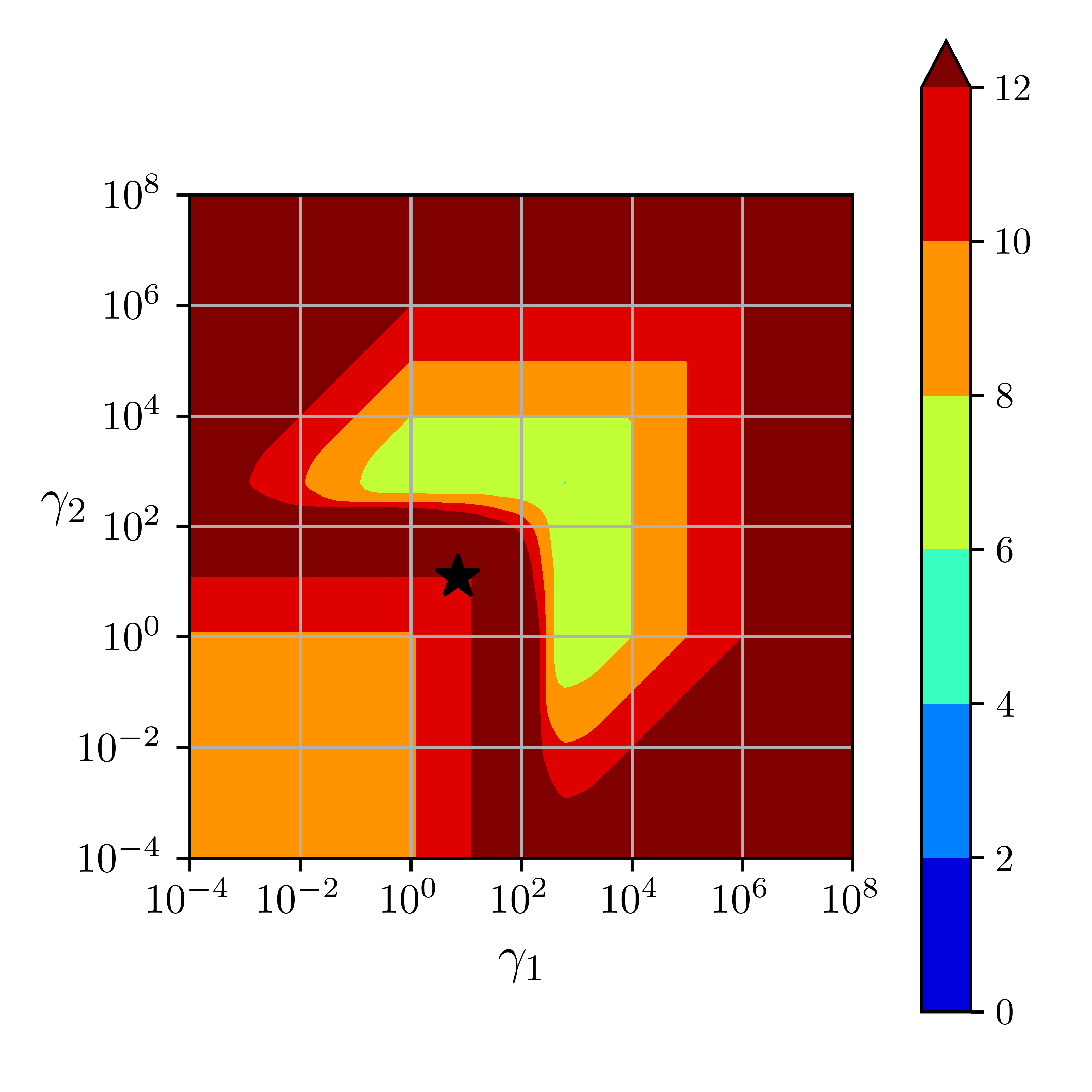}
	\vspace{-0.9cm}
	\caption{Baseline rational quadratic kernel: $\log(\kappa(\Sigmag))$}
	\label{Fig_cond_2d_Kern_RatQd_Kbase}
\end{subfigure}
\hspace{0.05cm}
\begin{subfigure}[t]{0.49\textwidth}
	\includegraphics[width=\textwidth]{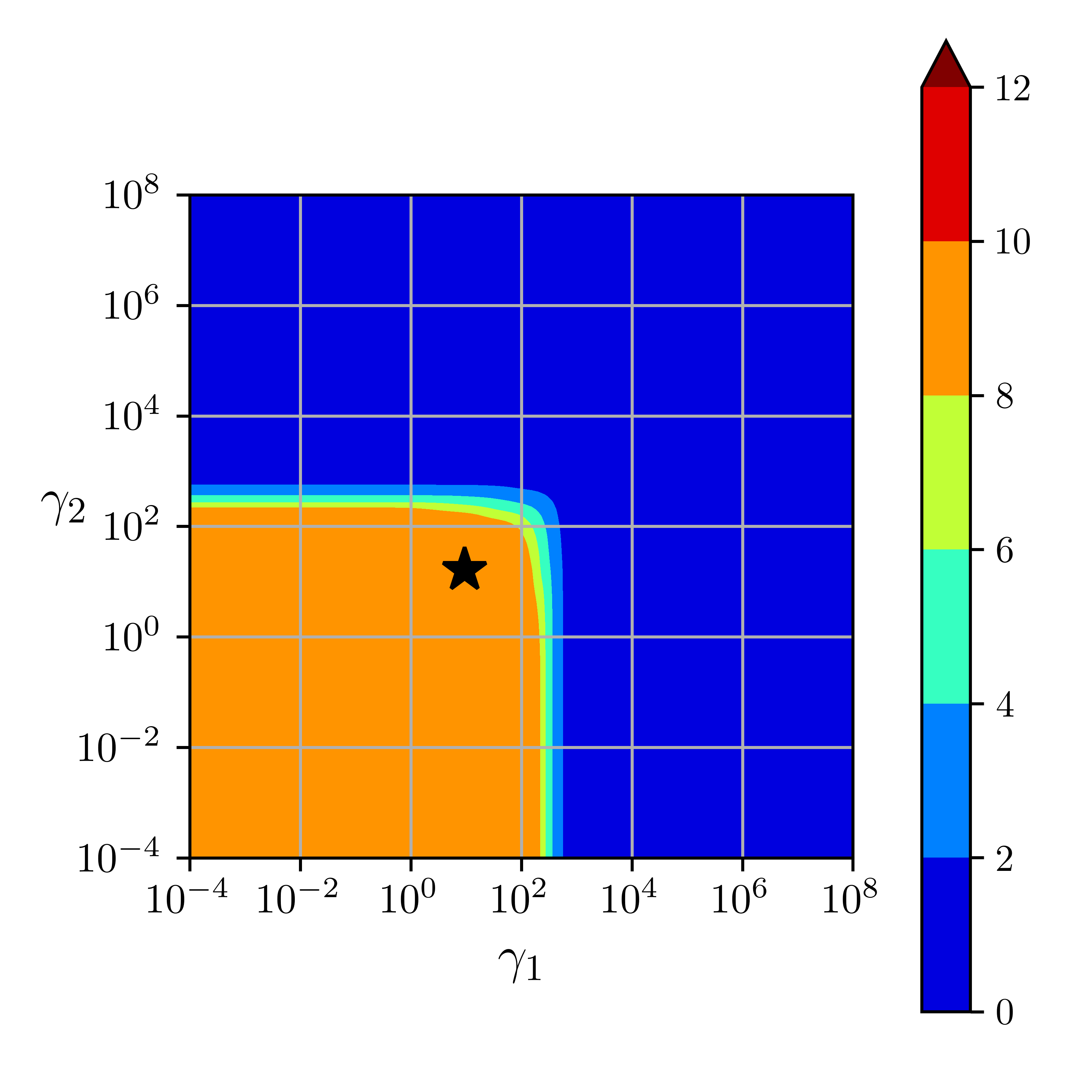} 
	\vspace{-0.9cm}
	\caption{Preconditioned rational quadratic kernel: $\log(\kappa(\Sigmagtil))$}
	\label{Fig_cond_2d_Kern_RatQd_precon}
\end{subfigure}
\caption{The condition number of the gradient-enhanced covariance matrix with the baseline and preconditioned methods. All cases use \Eq{Eq_X_dataset} for the set of evaluation points $\X$ and \Eq{Eq_etaKgtil} to calculate $\etaKgtil = 1.5 \times 10^{-9}$. The value of the hyperparameters $\gammavec$ that maximizes the marginal log-likelihood function from \Eq{Eq_ln_lkd_final} with the Rosenbrock function from \Eq{Eq_Rosenbrock} is indicated by the star marker.}
\label{Fig_cond_2d_Kern}
\end{figure}
%--------------------------------------------------------------------------

%--------------------
% New subsection
%--------------------
\subsection{Optimization} \label{Sec_Results_Optz}

In this section, the baseline, rescaling, and preconditioning methods are compared when used with a Bayesian optimizer to minimize the Rosenbrock function from \Eq{Eq_Rosenbrock} with $d \in \{2, 5, 10, 15 \}$. For each test case, the optimization is repeated five separate times for each method. The starting points are selected with the Latin hypercube sampling from the open source Surrogate Modeling Toolbox with a lower bound of $-10$ and an upper bound of $10$ for the parameters, and the random state set to 1. This ensures that the optimizer for each method starts from the same initial solution. The selected acquisition function is the upper-confidence function
\begin{equation}
	h(\xvec) = \mu_f(\xvec) + \omega \sigma_f(\xvec),
\end{equation}
where $\omega \geq 0$. The parameter $\omega$ promotes exploitation when it is small, and exploration when it is large. We use $\omega = 0$ since we are interested in local optimization for the unimodal Rosenbrock function. The gradient-based SLSQP optimizer from the Python library SciPy is used to select the hyperparameters by maximizing the marginal log-likelihood. The same optimizer is used to minimize the acquisition function to select the next point in the parameter space to evaluate the Rosenbrock function. A trust region is used in the minimization of the acquisition function, similar to the one used in \cite{eriksson_scalable_2019}, where a Bayesian optimizer was also used for local minimization. However, our trust region is set to be a hypersphere instead of a hyperrectangle.

%--------------------------------------------------------------------------
\begin{figure}[t]
\centering
	% Plots for the objective
\begin{subfigure}[t]{0.49\textwidth}
	\includegraphics[width=\textwidth]{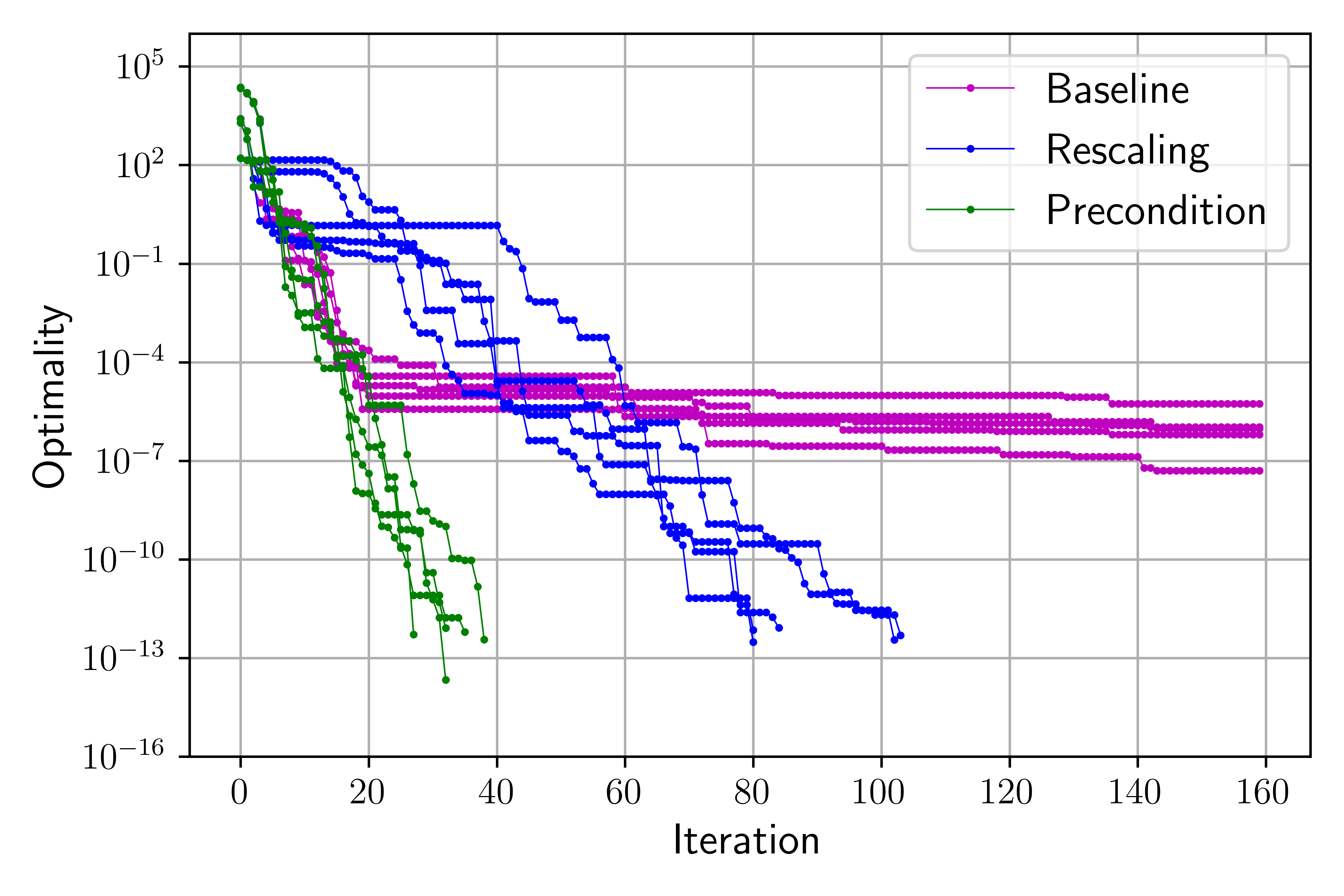}
	\caption{$d=2$}
	\label{Fig_optz_Rosen_opt_dim2}
\end{subfigure}
 \begin{subfigure}[t]{0.49\textwidth}
 	\includegraphics[width=\textwidth]{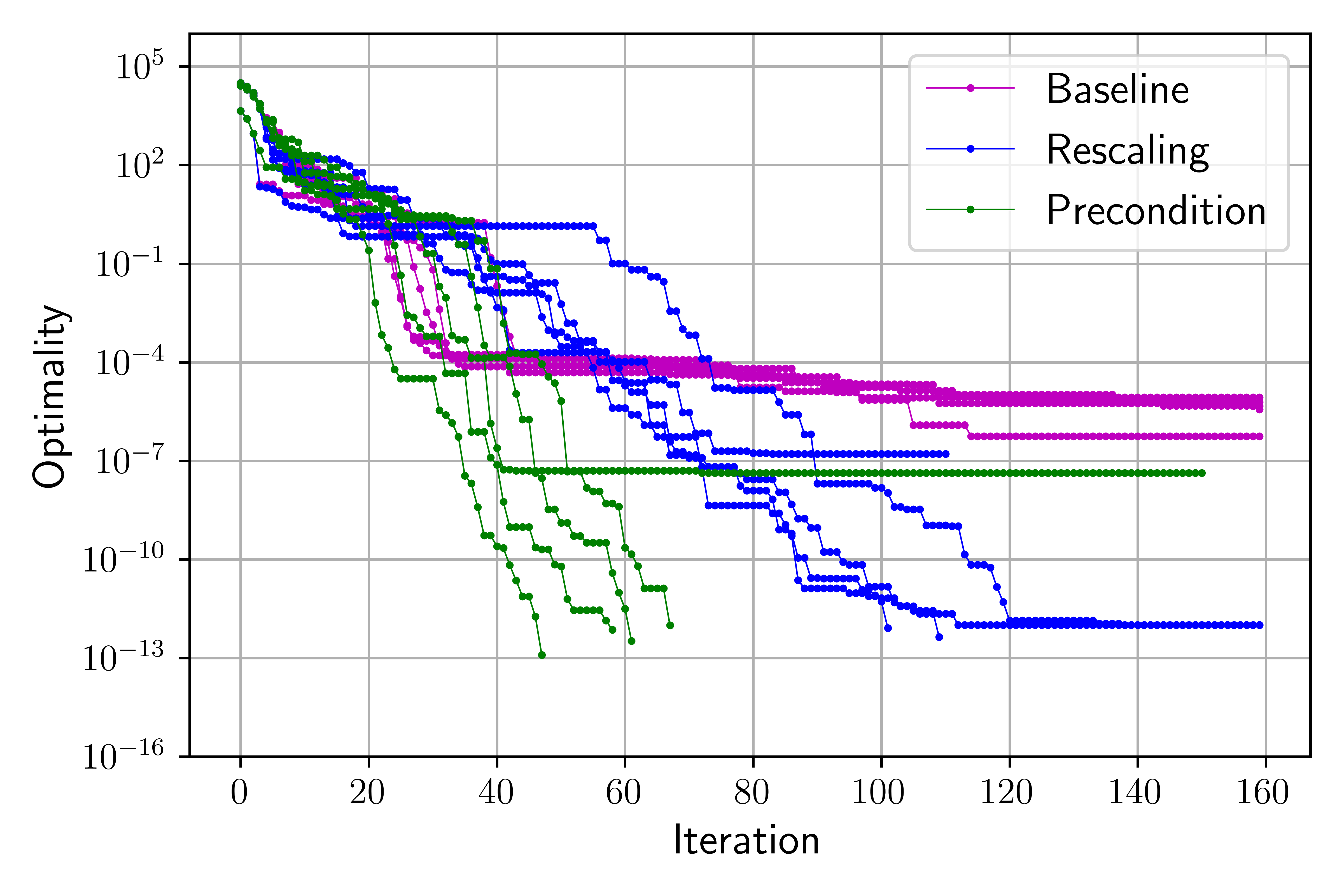}
	\caption{$d=5$}
	\label{Fig_optz_Rosen_opt_dim5}
\end{subfigure}
\hspace{0.1cm}
\begin{subfigure}[t]{0.49\textwidth}
	\includegraphics[width=\textwidth]{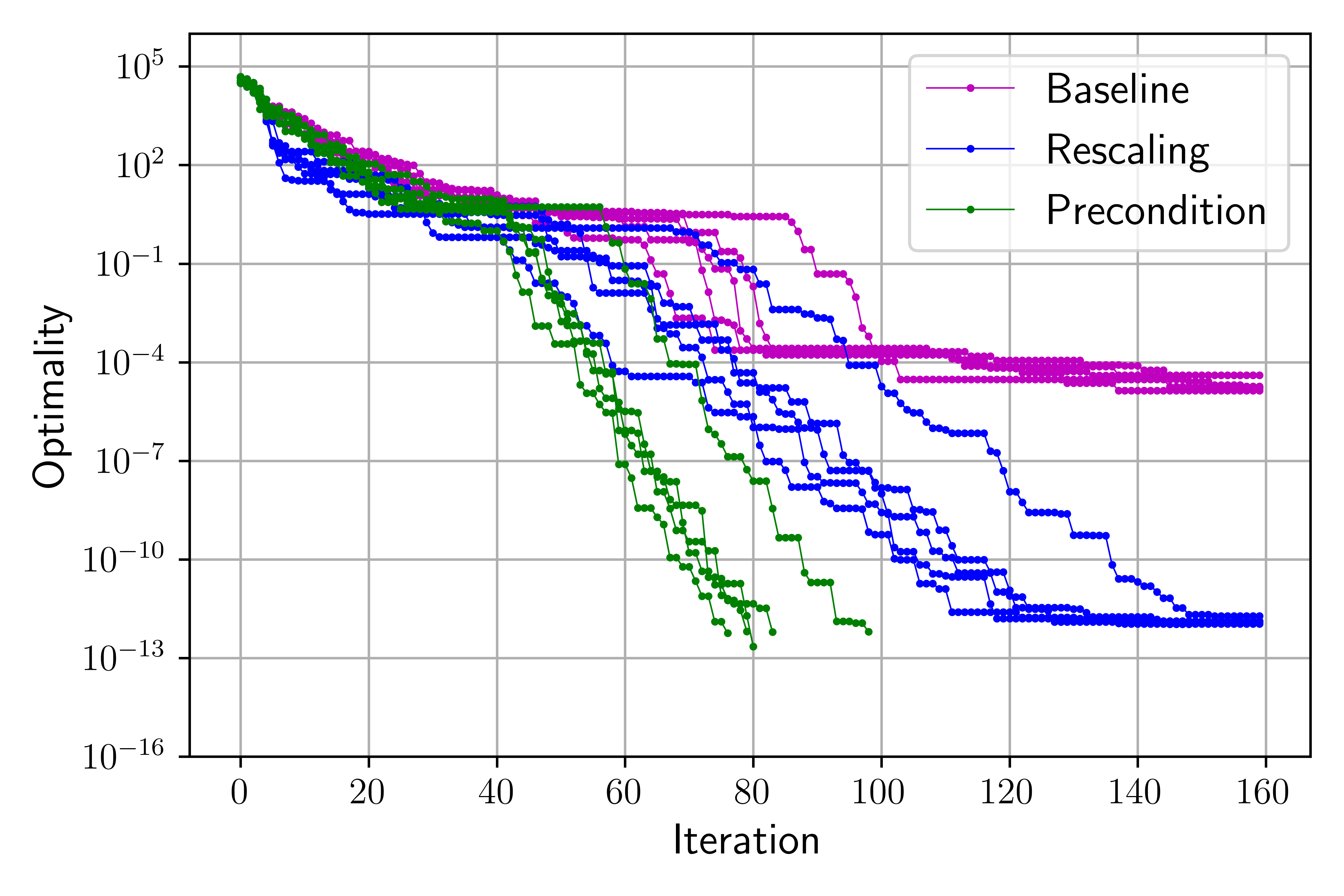}
	\caption{$d=10$}
	\label{Fig_optz_Rosen_opt_dim10}
\end{subfigure}	
%
% Plots for the optimality
\begin{subfigure}[t]{0.49\textwidth}
	\includegraphics[width=\textwidth]{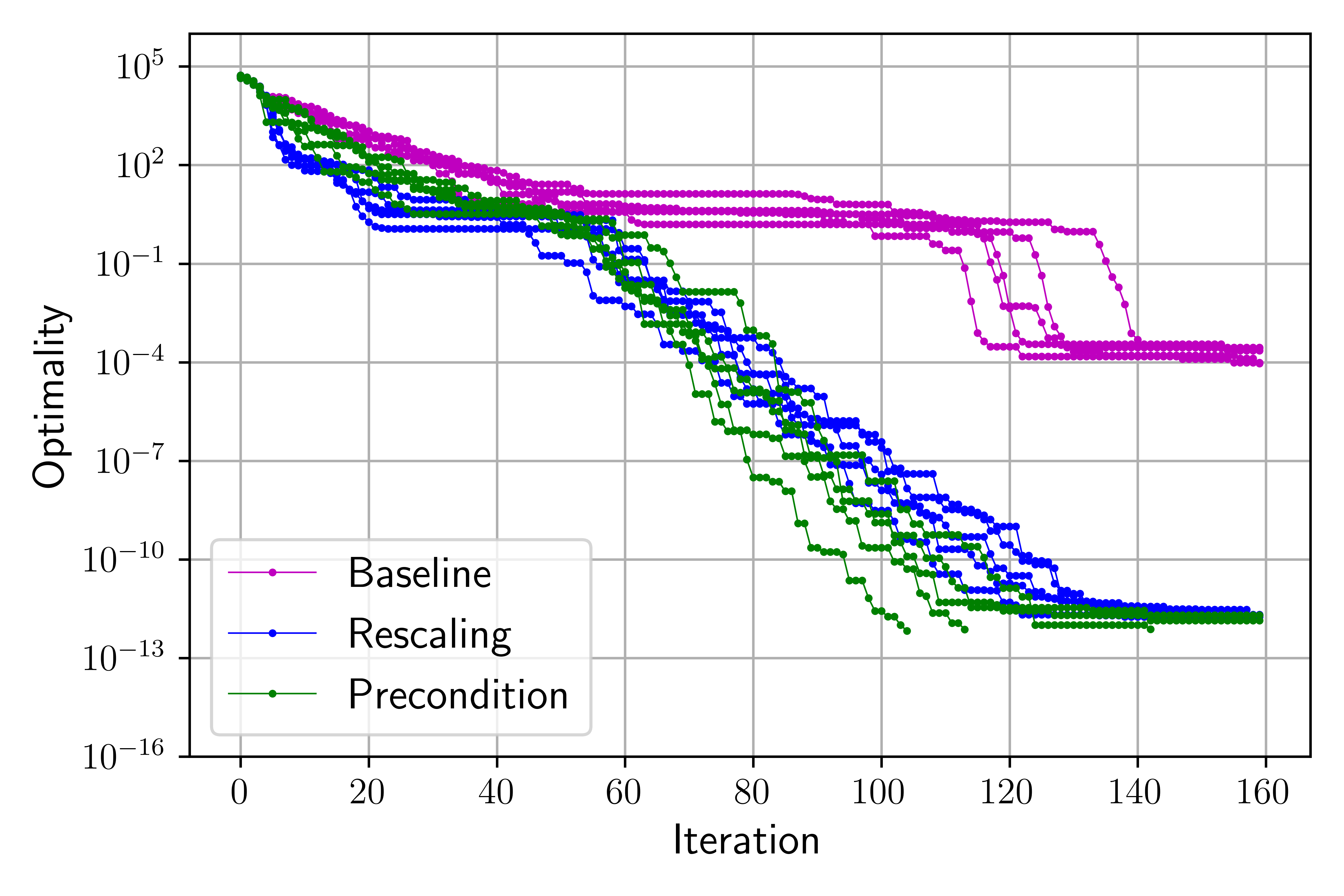}
	\caption{$d=15$}
	\label{Fig_optz_Rosen_opt_dim15}
\end{subfigure}
\caption{Bayesian optimization of the Rosenbrock function from \Eq{Eq_Rosenbrock} using the baseline, the rescaling, and preconditioning methods for $d \in \{2, 5, 10, 15 \}$. The plots show the lowest evaluated optimality, \ie the $\ell_2$ norm of the gradient of \Eq{Eq_Rosenbrock}, for each optimization run at each iteration.}
\label{Fig_optz_Rosen}
\end{figure}
%--------------------------------------------------------------------------

In \Fig{Fig_optz_Rosen} the optimality, which is the $\ell_2$ norm of the gradient, is compared for the Bayesian optimizer using the baseline, rescaling, and preconditioning methods. There are two important observations from these optimality plots: the depth and rate of convergence of the optimality. In all cases, the rescaling and preconditioning methods converge the optimality several orders of magnitude deeper than the baseline method. In fact, the rescaling and preconditioning methods converge the optimality to below $10^{-12}$ in all test cases. Meanwhile, the deepest optimality that the baseline method achieves is $10^{-7}$, and only for $d=2$. As the dimensionality increases, the optimizer with the baseline method is not able to converge the optimality as deeply and can only achieve an optimality of $10^{-4}$ for the $d=15$ case. The optimizer with the rescaling and preconditioning methods is thus able to converge the optimality 5 to 9 additional orders of magnitude relative to the optimizer with the baseline method. For the baseline method, the hyperparameters $\gammavec$ are selected by solving \Eq{Eq_optz_lkd}, where the marginal log-likelihood is maximized with an upper bound on the condition number. As the optimality is converged, the evaluation points get closer together in the parameter space and this makes the ill-conditioning of the gradient-enhanced covariance matrix worse \cite{marchildon_non-intrusive_2023}. Consequently, solving \Eq{Eq_optz_lkd} results in hyperparameters that provide a lower marginal log-likelihood since the upper bound on the condition number becomes a more onerous constraint. The rescaling and baseline methods do not suffer from this since, by construction, they guarantee that the selected hyperparameters maximize the marginal log-likelihood without being constrained by the condition number of the covariance matrix.

It is clear from \Fig{Fig_optz_Rosen} that the optimizer utilizing the preconditioning method achieves the fastest rate of convergence of the three methods for all four test cases. The convergence of the optimality for the optimizer with the rescaling method is significantly slower relative to the optimizer with the baseline and preconditioning methods, particularly for the cases with $d=2$ and $d=5$. The slower convergence of the optimality for the optimizer using the rescaling method was a trend that was also observed in \cite{marchildon_non-intrusive_2023}. This trend was found to be a consequence of the rescaling method providing a surrogate with gradients that have larger errors relative to the baseline method. 

In summary, the use of the preconditioning method with a gradient-enhanced Bayesian optimizer enables the optimality to be converged more deeply than with the use of the baseline method, and in fewer iterations than with the rescaling method.

%--------------------
% New section
%--------------------

\section{Conclusions} \label{Sec_Conclusions}

\begin{table}[t]
\centering
\begin{tabular}{l c c c} 
% \hline
Method & Baseline & Rescale & Precondition \\ [0.5ex] 
 \hline %\hline
 $\kappa(\Sigmag(\gammavec)) \leq \condmax, \forall \gammavec > 0$ & \xmark & $\gamma_1 = \ldots = \gamma_d$ & \cmark \\ 
 %\hline
 Constraint free hyperparameter optz & \xmark & \xmark & \cmark \\
 %\hline
 Nodes can be collocated & \cmark & \xmark & \cmark \\
 Deep convergence: optimality $< 10^{-8}$ & \xmark & \cmark & \cmark \\
 %\hline
 Provides a correlation matrix & \xmark & \xmark & \cmark \\
 %\hline
 Bounded $\kappa(\Sigmagtil(\gammavec))$ for other kernels & \xmark & \xmark & \cmark \\
 %\hline
\end{tabular}
\caption{Comparison of methods to address the ill-conditioning of the covariance matrix $\Sigmag$. The baseline and rescaling methods are summarized in \Secs{Sec_WellCondMtd_Baseline}{Sec_WellCondMtd_Rescale}, respectively, and the implementation of the preconditioning method is provided in \Sec{Sec_Implementation}.}
\label{Table_summary_methods}
\end{table}

A gradient-enhanced GP provides a more accurate probabilistic surrogate than its gradient-free counterpart but the ill-conditioning of its covariance matrix has been a hindrance to its use. A straightforward method has been developed that overcomes this problem and ensures that the condition number of the gradient-enhanced covariance matrix is always smaller than the user-set threshold of $\condmax$. The simple implementation is detailed in Algorithm \ref{Alg_Stable_Cho}, which is found in \Sec{Sec_Implementation}. The method applies a diagonal preconditioner along with a modest nugget that scales as $\etaKgtil = \mathcal{O}(n_x \sqrt{d})$ for the Gaussian kernel, and $\etaKgtil = \mathcal{O}(n_x d)$ for other kernels. A tighter bound for $\etaKgtil$ for the non-Gaussian kernels will be considered in future work. 

The benefits of using the preconditioning method relative to the baseline and rescaling methods are summarized in \Table{Table_summary_methods}. With the preconditioning method, all of the data points can be kept and there is no minimum distance requirement between evaluation points in the parameter space. The points can even be collocated, unlike the rescaling method. Since the preconditioning method ensures that $\bSigmagtil$, no constraint is required when maximizing the marginal log-likelihood. This simplifies the optimization and reduces its computational cost. The preconditioning method also provides a correlation matrix, which makes the GP easier to interpret.

In \Sec{Sec_Results_Optz} the Rosenbrock function was optimized for $d \in \{2, 5, 10, 15 \}$ with a Bayesian optimizer using the baseline, rescaling and preconditioning methods. The Bayesian optimizer with the preconditioning method converged the optimality an additional 5-9 orders of magnitude relative to the optimizer with the baseline method. Furthermore, the preconditioning method enabled the Bayesian optimizer to converge the optimality more quickly than when the rescaling method is used, particularly for the lower-dimensional problems. The slower convergence of a Bayesian optimizer using the rescaling method was previously identified to be the result of its surrogate having gradients with larger errors. In conclusion, the preconditioning method bounds the condition number of the preconditioned gradient-enhanced covariance matrix and it enables a Bayesian optimizer to achieve deeper and faster convergence relative to the use of either the baseline or rescaling methods.

The baseline, rescaling, and preconditioning methods are all available in the open source python library GpGradPy, which can be found at \url{https://github.com/marchildon/gpgradpy/tree/paper_precon}. All of the figures in this paper can be reproduced with this library.

\appendix

%--------------------
% New section
%--------------------

\section{Proofs} \label{Sec_ApxProof}

%--------------------
% New section
%--------------------

\subsection{Proof for \Prop{Prop_urb}} \label{Sec_ApxProof_Prop_urb}

The derivation of the upper bound for the sum of the absolute value of the off-diagonal entries is the same for any of the last $n_x d$ rows of $\Kgtil$, which comes from \Eq{Eq_Kgtil}. Without loss of generality, we consider the $b$-th row of $\Kgtil$, where $b=p n_x + m$, and $p$ and $m$ can take any integer values that satisfy $1 \leq p \leq d$ and $1 \leq m \leq n_x$:
\begin{align*}
	\sum_{\substack{i=1 \\ i\neq b}}^{n_x} \left( \Kgtil \right)_{bi}
		&= \sum_{\substack{i=1 \\ i \neq m}}^{n_x} 
		\left( \left| \xtil_{m p} - \xtil_{i p} \right| 
			+ \sum_{j=1}^{d} \left| \delta_{j p} - \left( \xtil_{m p} - \xtil_{i p} \right) \left( \xtil_{m j} - \xtil_{i j} \right) \right| \right) e^{-\frac{ \| \xvectil_{m :} - \xvectil_{i:} \|_2^2}{2}} \\
		& \leq \sum_{\substack{i=1 \\ i \neq m}}^{n_x} \left( \left| \xtil_{m p} - \xtil_{i p} \right|
			+ 1 + \sum_{\substack{j= 1\\ j \neq p}}^{d} \left| \left( \xtil_{m p} - \xtil_{i p} \right) \left( \xtil_{m j} - \xtil_{i j} \right) \right| \right) e^{-\frac{ \| \xvectil_{m :} - \xvectil_{i:} \|_2^2}{2}} \\
		&= \sum_{\substack{i=1 \\ i \neq m}}^{n_x} \left( 1 + \left| \xtil_{m p} - \xtil_{i p} \right| \left(1 +
		\sum_{\substack{j= 1\\ j\neq p}}^{n_x} \left| \xtil_{m j} - \xtil_{i j} \right| \right) \right) e^{-\frac{ \| \xvectil_{m :} - \xvectil_{i:} \|_2^2}{2}} \\
		&\leq (n_x - 1) \max_{\nu \geq 0, \, \check{\wvec} \geq 0} \left(1 + \nu \left( 1 + \check{\one}_p^\top \check{\wvec} \right) \right) e^{-\frac{1}{2} \left(\nu^2 + \check{\wvec}^\top \check{\wvec} \right) }, \yesnumber \label{Eq_urb_max}
\end{align*} 
where $\nu = |\xtil_{m p} - \xtil_{i p} |$ and $\check{w}_j = |\xtil_{m j} - \xtil_{i j} |$, except for $\check{w}_p = 0$. Similarly $\check{\one}_p$ is a vector of ones of length $d$ with a zero at its $p$-th entry. The first inequality is a result of $\Kgtil$ being a correlation matrix, as explained in \Sec{Sec_GpNew}, and therefore $|\pp{\K}{\xtil_a}{\ytil_a}| \leq 1$. 

An analogous approach to the one taken in \Prop{Prop_ura} can be used to show that the maximization of \Eq{Eq_urb_max} requires $\check{\wvec} = \alpha \check{\one}_p$, \ie that all but the $p$-th entries in $\check{\wvec}$ are equal. Using $\check{\wvec} = \alpha \check{\one}_p$ with \Eq{Eq_urb_max} gives
\begin{equation} \label{Eq_urb_g1}
	g_1(\nu, \alpha; d) 
		= \left(\nu + 1 + (d-1) \alpha \nu \right) e^{-\frac{\nu^2 + (d-1) \alpha^2}{2}}.
\end{equation}
We thus need to prove that $(n_x - 1 ) g_1(\nu, \alpha; d) < u_{\text{G}}(n_x,d)$ for $\nu, \alpha \geq 0$ and $d \in \mathbb{Z}^+$. The following lemma considers the case for $d=1$.

\begin{lemma} \label{Lem_g1_max_d1}
	For $d=1$ we have 
\begin{equation} \label{Eq_urb_d1}
	(n_x - 1) \max_{\nu \geq 0, \alpha \geq 0} g_1(\nu, \alpha; d=1) = u_{\text{G}}(n_x,d=1) = (n_x - 1) \left( 1 + \sqrt{5} \right) e^{-\frac{3 - \sqrt{5}}{4}} ,
\end{equation}
where $g_1(\nu, \alpha; d)$ comes from \Eq{Eq_urb_g1} and $u_{\text{G}}(n_x,d)$ comes from \Eq{Eq_ub_Kgtil_G}.

\end{lemma} 

\begin{proof}
{
For $d=1$ the parameter $\alpha$ cancels out and we thus have a scalar function that we seek to maximize
\begin{align*}
%	g_1(\nu, \alpha; d=1) 
%		&= \left(\nu + 1 \right) e^{-\frac{\nu^2}{2}} \\
	\p{g_1(\nu;d=1)}{\nu} 
		&= \p{\left( \left(\nu + 1 \right) e^{-\frac{\nu^2}{2}} \right)}{\nu} \\
		&= -\left( \nu^2 + \nu -1\right) e^{-\frac{\nu^2}{2}} = 0 \\
	\nu^*_{d=1} 
		&= \frac{-1 + \sqrt{5}}{2},
\end{align*}
where only the positive root was kept since $\nu \geq 0$ and it is straightforward to verify that this critical point maximizes $g_1(\nu, \alpha; d=1)$. \Eq{Eq_urb_d1} is recovered by evaluating $g_1$ with $\nu = \nu^*_{d=1}$ and $d=1$, which completes the proof.
}
\end{proof}

To consider the cases for $d \geq 2$ we will need to find the value of $\alpha$ and $\nu$ that maximize $g_1(\nu, \alpha; d)$ from \Eq{Eq_urb_g1}. The following lemma considers the maximization of $g_1$ with respect to $\alpha$.

\begin{lemma} \label{Lem_g2_geq_g1}
	For $\nu, \alpha \geq 0$ and $d \geq 2$ we have $g_1(\nu, \alpha; d) \leq g_2(\nu; d)$, where $g_1$ comes from \Eq{Eq_urb_g1} and 
\begin{equation} \label{Eq_urb_g2}
	g_2(\nu; d)
		= \left( \frac{\nu + 1 + \sqrt{h_1(\nu;d)}}{2} \right) e^{-\frac{\nu^2}{2} + h_2(\nu;d)},
\end{equation}
where 
\begin{align} 
	h_1(\nu;d) &= (\nu+1)^2 + 4 \nu^2 (d-1) \label{Eq_urb_h1} \\
	h_2(\nu;d) &= \frac{(\nu + 1) \sqrt{h_1(\nu;d)} - (\nu+1)^2}{4 \nu^2 (d-1)} - \frac{1}{2}. \label{Eq_urb_h2}
\end{align}
	
\end{lemma}

\begin{proof}

To find the maximum of $g_1(\nu, \alpha;d)$ with respect to $\alpha$ we find its derivative, set it to zero and solve for $\alpha$:
\begin{align*} 
	\p{g_1(\nu, \alpha; d)}{\alpha} 
		&= -(d-1) \left( (d-1) \nu \alpha^2 + (\nu+1) \alpha - \nu \right) e^{-\frac{\nu^2 + (d-1) \alpha^2}{2}} = 0 \\
	\alpha^*
		&= \frac{-(\nu + 1) + \sqrt{(\nu + 1)^2 + 4 \nu^2 (d-1) }}{2 \nu (d-1)},
\end{align*} 
where only the positive root of the quadratic equation is kept since $\alpha$ must be positive and it is straightforward to verify that this provides the maximum of $g_1(\nu, \alpha; d)$. The function $g_2(\nu;d)$ from \Eq{Eq_urb_g2} is recovered by evaluating $g_1(\nu, \alpha^*;d)$, which completes the proof.

\end{proof}

Both $h_1(\nu;d)$ and $h_2(\nu;d)$ from \Eqs{Eq_urb_h1}{Eq_urb_h2}, respectively, are non-polynomial functions that make it impractical to find a closed-form maximum solution for $g_2(\nu;d)$. The two following lemmas provide upper bounds for these non-polynomial functions. 

\begin{lemma} \label{Lem_ub_h1}

For $\nu \geq 0$ and $d \geq 2$ we have the bound $\sqrt{h_1(\nu;d)} \leq h_3(\nu;d)$, where $h_1(\nu;d)$ comes from \Eq{Eq_urb_h1} and $h_3(\nu;d)$ is the following $\mathbb{C}^0$ continuous piecewise polynomial:
\begin{align}
	h_3(\nu;d) =
	\begin{cases}
		(2 \sqrt{d} -1) \nu +1 	& \text{if } 0 \leq \nu \leq 1 \\ 
		2 \sqrt{d} \nu 				& \text{if } \nu > 1. \label{Eq_urb_h3}
	\end{cases}
\end{align}

\end{lemma}

\begin{proof}
{ 
For $0 \leq \nu \leq 1$ and $d \geq 2$ we start by showing that $h_3^2 \geq h_1$:
\begin{align*}
	\left( (2 \sqrt{d} -1) \nu + 1 \right)^2 
		& \geq (\nu + 1)^2 + 4 \nu^2 (d-1) \\
	4 \nu (1 - \nu) (\sqrt{d} - 1)
		& \geq 0.
\end{align*}
Next we demonstrate that $h_3^2 \geq h_1(\nu;d)$ for $\nu \geq 1$:
\begin{align*}
	\left( 2 \sqrt{d} \nu \right)^2
		& \geq (\nu + 1)^2 + 4 \nu^2 (d-1) \\
	\left( \nu + \frac{1}{3} \right) \left( \nu -1 \right)
		& \geq 0. 
\end{align*}
Finally, it is straightforward to verify that $h_3(\nu;d)$ is $\mathbb{C}^0$ continuous:
\begin{equation}
	\lim_{\nu \rightarrow 1^{-}} h_3(\nu;d) = \lim_{\nu \rightarrow 1^{+}} h_3(\nu;d) = 2 \sqrt{d},
\end{equation}
which completes the proof. 
}
\end{proof}

\begin{lemma} \label{Lem_ub_h2}
The maximum value for $h_2(\nu;d)$ from \Eq{Eq_urb_h2} for $\nu \geq 0$ and $d \in \mathbb{Z}^+$ is
\begin{equation} \label{Eq_max_exp_fun}
	\max_{\nu \geq 0} h_2(\nu;d) = \lim{\nu \rightarrow 0} h_2(\nu, d) = 0.
\end{equation}
\end{lemma}

\begin{proof}
{
We start by proving that $h_2(\nu;d)$ is monotonically decreasing with respect to $\nu$ by showing that its derivative with respect to $\nu$ is nonpositive for $\nu \geq 0$ and $d \in \mathbb{Z}^+$
\begin{align*}
	\p{h_2(\nu;d)}{\nu} 
		&= -\frac{h_1 - (\nu +1) \sqrt{h_1} - 2 \nu^2 (d-1)}{2 \nu^3 \sqrt{h_1} (d-1)}.
\end{align*}
Since the denominator of $\p{h_2(\nu;d)}{\nu} $ is always nonnegative for $\nu \geq 0$ and $d \in \mathbb{Z}^+$, we only need to show that its numerator is nonpositive for the same range of parameters:
\begin{align*}
	h_1 - (\nu +1) \sqrt{h_1} - 2 \nu^2 (d-1) 
		&\geq 0 \\
	\left( h_1 - 2 \nu^2 (d-1) \right)^2 
		&\geq (\nu +1)^2 a \\
	4 \nu^4 (d-1)^2 
		&\geq 0.
\end{align*}
Since $h_2(\nu;d)$ is monotonically decreasing with respect to $\nu$ for $\nu \geq 0$ and $d \in \mathbb{Z}^+$, its maximum is at $\nu=0$. To evaluate $h_2(\nu;d)$ we use a limit and apply l'H\^opital's rule twice:
\begin{align*}
	\lim_{\nu \rightarrow 0} h_2(\nu;d)
		&= \lim_{\nu \rightarrow 0} \frac{(\nu+1) \sqrt{h_1} - (\nu+1)^2}{4(d-1) \nu^2} - \frac{1}{2} \\
		&= \lim_{\nu \rightarrow 0} \frac{2 \sqrt{h_1} + \frac{4 \nu (d-1)}{\sqrt{h_1}} - 2(\nu+1)}{8 \nu (d-1)} - \frac{1}{2} \\
		&= \lim_{\nu \rightarrow 0} \frac{\frac{4(\nu+1)(d-1)}{h_1^{\frac{3}{2}}} + \frac{(8d-6) \nu + 2}{\sqrt{h_1}} -2}{8(d-1)} - \frac{1}{2} \\
		&= 0,
\end{align*}
which completes the proof. 
}
\end{proof}

Thanks to \Lems{Lem_ub_h1}{Lem_ub_h2} it is now possible to derive a closed-form solution for an upper bound of $g_2(\nu;d)$ from \Eq{Eq_urb_g2} for $\nu \geq 0$ and $d \geq 2$. This is considered in the two following lemmas that consider the case for $0 \leq \nu \leq 1$ and $\nu > 1$, respectively. 

\begin{lemma} \label{Lem_g3_geq_g2}
	For $0 \leq \nu \leq 1$ and $d \geq 2$ we have $(n_x - 1) g_2(\nu;d) < u_{\text{G}}(n_x,d)$, where $g_2(\nu;d)$ and $u_{\text{G}}(n_x,d)$ come from \Eqs{Eq_urb_g2}{Eq_ub_Kgtil_G}, respectively.
\end{lemma}

\begin{proof}
	The function $g_2(\nu;d)$ from \Eq{Eq_urb_g2} contains the nonlinear functions $h_1(\nu;d)$ and $h_2(\nu;d)$ from \Eqs{Eq_urb_h1}{Eq_urb_h2}, respectively. We use the upper bounds provided by \Lems{Lem_ub_h1}{Lem_ub_h2} for these functions and $0 \leq \nu \leq 1$ to get $g_2(\nu;d) < g_3(\nu;d)$, where
\begin{align}
	g_3(\nu;d) 
		&= \frac{\nu + 1 + \left[(2 \sqrt{d} - 1) \nu + 1 \right]}{2} e^{-\frac{\nu^2}{2}} \nonumber \\
		&= \left(\sqrt{d} \nu + 1 \right) e^{-\frac{\nu^2}{2}}. \label{Eq_urb_g3} 
\end{align}
We now find the value of $\nu$ that maximizes $g_3(\nu;d)$
\begin{align*}
	\p{g_3}{\nu} 
		&= \left( \sqrt{d} - \nu \left( \sqrt{d} \nu +1 \right) \right) e^{-\frac{\nu^2}{2}} = 0 \\
	\nu_3^*
		&= \frac{-1 + \sqrt{1 + 4d}}{2 \sqrt{d}}, \yesnumber \label{Eq_nu_star_3}
\end{align*} 
where only the positive root was kept and it is straightforward to verify that $0 < \nu^* < 1$ for $d \geq 2$, and that this is the maximum for the function $g_3$. Using $\nu_3^*$ from \Eq{Eq_nu_star_3} gives $g_3(\nu_3^*; d) = u_{\text{G}}(n_x,d)$, where $u_{\text{G}}(n_x,d)$ comes from \Eq{Eq_ub_Kgtil_G}. Therefore, we have for $d \geq 2$:
\begin{equation}
	\max_{0 \leq \nu \leq 1} g_2(\nu; d) < \max_{0 \leq \nu \leq 1} g_3(\nu; d) = g_3(\nu_3^*; d) = u_{\text{G}}(n_x,d),
\end{equation}
which completes the proof.
	
\end{proof}

\begin{lemma} \label{Lem_g4_geq_g2}
	We have $(n_x - 1) g_2(\nu;d) < u_{\text{G}}(n_x,d)$ for $\nu \geq 1$ and $d \geq 2$, where $g_2(\nu;d)$ and $u_{\text{G}}(n_x,d)$ come from \Eqs{Eq_urb_g2}{Eq_ub_Kgtil_G}, respectively.
\end{lemma}

\begin{proof}
{
We now consider the case for $\nu > 1$ by substituting $h_3(\nu;d)$ from \Eq{Eq_urb_h3} for $\nu > 1$ into $g_2(\nu;d)$ for $\sqrt{h_1(\nu;d)}$ and using the results from \Lem{Lem_ub_h2} for an upper bound on $h_2(\nu;d)$. We get the bound $g_2(\nu;d) \leq g_4(\nu;d)$, where
\begin{equation}
	g_4(\nu;d) = \frac{\nu + 1 + \left[2 \sqrt{d} \nu \right]}{2} e^{-\frac{\nu^2}{2}}.
\end{equation}
We now find the value of $\nu \geq 1$ that maximizes $g_4(\nu;d)$:
\begin{align*}
	\p{g_4}{\nu} 
		&= \frac{ \left( 2 \sqrt{d} + 1 \right) - \nu \left( \left( 2 \sqrt{d} + 1 \right) \nu +1 \right)}{2} e^{-\frac{\nu^2}{2}} = 0\\
	(2 \sqrt{d} +1) \nu^2 + \nu - (2 \sqrt{d} +1)
		&= 0 \\
	\nu_4^* 
		&= \frac{ -1 + \sqrt{1 + 4 \left(2 \sqrt{d} +1\right)^2 }}{2(2 \sqrt{d} + 1)},
\end{align*}
where only the positive root was kept and it is straightforward to show that this provides the maximum for $g_4(\nu;d)$. However, we now demonstrate that this root does not satisfy the constraint $\nu \geq 1$:
\begin{equation*}
	\nu_4^* < \frac{ -1 + \left[ 1 + 2 \left(2 \sqrt{d} +1\right)\right]}{2(2 \sqrt{d} + 1)}
		= 1,
\end{equation*}
where we used the inequality $\sqrt{b_1 +b_2} < \sqrt{b_1} + \sqrt{b_2}$ for $b_1, b_2 > 0$. Since there are no roots for $\nu \geq 1$ that maximize $g_4(\nu;d)$ for $d \geq 2$, it is either maximized at $\nu = 1$ or $\nu \rightarrow \infty$. For $\lim \nu \rightarrow \infty$ we have $g_4(\nu;d) = 0$ and for $\nu = 1$ we have 
\begin{equation}
	(n_x -1 ) g_4(\nu = 1, d) = (n_x -1 ) g_3(\nu =1, d) < (n_x -1 ) g_3(\nu_3^*, d) = u_{\text{G}}(n_x,d),
\end{equation}
where $g_3 = g_4$ for $\nu = 1$ since both functions used the relation $\sqrt{h_1(\nu;d)} \leq h_3(\nu;d)$ and it was shown in \Lem{Lem_ub_h1} that $h_3(\nu)$ from \Eq{Eq_urb_h3} is $\mathbb{C}^0$ continuous. We thus have $(n_x -1 ) g_2(\nu;d) < u_{\text{G}}(n_x,d)$ for $\nu \geq 1$ and $d \geq 2$, which completes the proof.
}
\end{proof}

It has been proven that the function $g_1$ from \Eq{Eq_urb_g1}, which provides an upper bound for the sum of absolute values for the off-diagonal entries for any of the last $n_x d$ rows of $\Kgtil$, is smaller than $u_{\text{G}}(n_x,d)$ for $n_x, d \in \mathbb{Z}^+$, which completes the proof.

%--------------------
% New section
%--------------------

\subsection{Proof for \Lem{Lem_trend_etaKgtil}} \label{Sec_ApxProof_Lem_trend_etaKgtil}

We start by deriving an upper bound for the exponent for $\etaKgtil$ in \Eq{Eq_etaKgtil}. To do this we take the derivative of the exponent, which we denote as $g$, with respect to $d$:
\begin{align*}
	\p{g}{d}
		&= \frac{\sqrt{1 + 4d} - 1 - 2d}{4d^2 \sqrt{1 + 4d}} \\
		&\leq \frac{\left[1 + 2 \sqrt{d}\right] - 1 - 2d}{4d^2 \sqrt{1 + 4d}} \\
		&= \frac{ \sqrt{d}- d}{2d^2 \sqrt{1 + 4d}},
\end{align*}
which is always negative for $d \geq 1$. Therefore, $g(d)$ is monotonically decreasing with respect to $d$ for $d \in \mathbb(Z)^+$ and is thus maximized at $g(d=1) = -\frac{3 - \sqrt{5}}{4}$. An upper bound for $\etaKgtil$ from \Eq{Eq_etaKgtil} is now derived

\begin{align*}
	\etaKgtil(n_x, d) 
		&= \frac{1 + (n_x-1) \frac{1 + \sqrt{1 + 4d}}{2} e^{-\frac{1 + 2d - \sqrt{1 + 4d}}{4d}}}{\condmax - 1} \\
		&< \frac{1 + (n_x-1) \frac{1 + \left[\sqrt{1} + \sqrt{4d} \right]}{2} e^{-\frac{3 - \sqrt{5}}{4}}}{\condmax - 1} \\
		&= \frac{1 + (n_x - 1)(1 + \sqrt{d}) e^{-\frac{3 - \sqrt{5}}{4} }}{\condmax -1},
\end{align*}
where it is clear that $\etaKgtil(n_x, d) = \mathcal{O}(n_x \sqrt{d})$, which completes the proof.

\begin{acknowledgements}

The authors would like to thank the Natural Sciences and Engineering Research Council of Canada and the Ontario Graduate Scholarship Program for their financial support. 

\end{acknowledgements}

\section*{Conflict of interest}

The authors declare that they have no conflict of interest.

\bibliographystyle{spmpsci}      % mathematics and physical sciences
\bibliography{MyLibrary.bib}   % name your BibTeX data base

\begin{thebibliography}{10}
\providecommand{\url}[1]{{#1}}
\providecommand{\urlprefix}{URL }
\expandafter\ifx\csname urlstyle\endcsname\relax
  \providecommand{\doi}[1]{DOI~\discretionary{}{}{}#1}\else
  \providecommand{\doi}{DOI~\discretionary{}{}{}\begingroup
  \urlstyle{rm}\Url}\fi

\bibitem{ababou_condition_1994}
Ababou, R., Bagtzoglou, A.C., Wood, E.F.: On the condition number of covariance
  matrices in kriging, estimation, and simulation of random fields.
\newblock Mathematical Geology \textbf{26}(1), 99--133 (1994).
\newblock \doi{10.1007/BF02065878}

\bibitem{ameli_noise_2022}
Ameli, S., Shadden, S.C.: Noise {{Estimation}} in {{Gaussian Process
  Regression}}.
\newblock arXiv p.~41 (2022)

\bibitem{dalbey_efficient_2013}
Dalbey, K.: Efficient and robust gradient enhanced {{Kriging}} emulators.
\newblock Tech. Rep. SAND2013-7022, 1096451 (2013).
\newblock \doi{10.2172/1096451}

\bibitem{davis_six_1997}
Davis, G.J., Morris, M.D.: Six {{Factors Which Affect}} the {{Condition
  Number}} of {{Matrices Associated}} with {{Kriging}}.
\newblock Mathematical Geology \textbf{29}(5), 669--683 (1997).
\newblock \doi{10.1007/BF02769650}

\bibitem{de_roos_high-dimensional_2021}
De~Roos, F., Gessner, A., Hennig, P.: High-{{Dimensional Gaussian Process
  Inference}} with {{Derivatives}}.
\newblock In: 38th {{International Conference}} on {{Machine Learning}}, pp.
  2535--2545 (2021)

\bibitem{eriksson_scaling_2018}
Eriksson, D., Dong, K., Lee, E., Bindel, D., Wilson, A.G.: Scaling {{Gaussian
  Process Regression}} with {{Derivatives}}.
\newblock In: 32nd {{Conference}} on {{Neural Information Processing Systems}}.
  {Montreal, Canada} (2018)

\bibitem{eriksson_scalable_2019}
Eriksson, D., Pearce, M., Gardner, J., Turner, R.D., Poloczek, M.: Scalable
  {{Global Optimization}} via {{Local Bayesian Optimization}}.
\newblock In: 33rd {{Conference}} on {{Neural Information Processing Systems}},
  p.~12. {Vancouver, Canada} (2019)

\bibitem{han_improving_2013}
Han, Z.H., G{\"o}rtz, S., Zimmermann, R.: Improving variable-fidelity surrogate
  modeling via gradient-enhanced kriging and a generalized hybrid bridge
  function.
\newblock Aerospace Science and Technology \textbf{25}(1), 177--189 (2013).
\newblock \doi{10.1016/j.ast.2012.01.006}

\bibitem{he_instability_2018}
He, X., Chien, P.: On the {{Instability Issue}} of {{Gradient-Enhanced Gaussian
  Process Emulators}} for {{Computer Experiments}}.
\newblock SIAM/ASA Journal on Uncertainty Quantification \textbf{6}(2),
  627--644 (2018).
\newblock \doi{10.1137/16M1088247}

\bibitem{higham_cholesky_2009}
Higham, N.J.: Cholesky factorization.
\newblock Wiley Interdisciplinary Reviews: Computational Statistics
  \textbf{1}(2), 251--254 (2009).
\newblock \doi{10.1002/wics.18}

\bibitem{hung_random_2021}
Hung, T.H., Chien, P.: A {{Random Fourier Feature Method}} for {{Emulating
  Computer Models With Gradient Information}}.
\newblock Technometrics \textbf{63}(4), 500--509 (2021).
\newblock \doi{10.1080/00401706.2020.1852973}

\bibitem{kostinski_condition_2000}
Kostinski, A.B., Koivunen, A.C.: On the condition number of {{Gaussian}}
  sample-covariance matrices.
\newblock IEEE Transactions on Geoscience and Remote Sensing \textbf{38}(1),
  329--332 (2000).
\newblock \doi{10.1109/36.823928}

\bibitem{laurent_overview_2019}
Laurent, L., Le~Riche, R., Soulier, B., Boucard, P.A.: An {{Overview}} of
  {{Gradient-Enhanced Metamodels}} with {{Applications}}.
\newblock Archives of Computational Methods in Engineering \textbf{26}(1),
  61--106 (2019).
\newblock \doi{10.1007/s11831-017-9226-3}

\bibitem{march_gradient-based_2011}
March, A., Willcox, K., Wang, Q.: Gradient-based multifidelity optimisation for
  aircraft design using {{Bayesian}} model calibration.
\newblock The Aeronautical Journal \textbf{115}(1174), 729--738 (2011).
\newblock \doi{10.1017/S0001924000006473}

\bibitem{marchildon_non-intrusive_2023}
Marchildon, A.L., Zingg, D.W.: A {{Non-intrusive Solution}} to the
  {{Ill-Conditioning Problem}} of the {{Gradient-Enhanced Gaussian Covariance
  Matrix}} for {{Gaussian Processes}}.
\newblock Journal of Scientific Computing \textbf{95}(3), 65 (2023).
\newblock \doi{10.1007/s10915-023-02190-w}

\bibitem{mohammadi_analytic_2017}
Mohammadi, H., Riche, R.L., Durrande, N., Touboul, E., Bay, X.: An analytic
  comparison of regularization methods for {{Gaussian Processes}} (2017)

\bibitem{ollar_gradient_2017}
Ollar, J., Mortished, C., Jones, R., Sienz, J., Toropov, V.: Gradient based
  hyper-parameter optimisation for well conditioned kriging metamodels.
\newblock Structural and Multidisciplinary Optimization \textbf{55}(6),
  2029--2044 (2017).
\newblock \doi{10.1007/s00158-016-1626-8}

\bibitem{osborne_gaussian_2009}
Osborne, M.A., Garnett, R., Roberts, S.J.: Gaussian {{Processes}} for {{Global
  Optimization}}.
\newblock In: 3rd {{International Conference}} on {{Learning}} and
  {{Intelligent Optimization}}. {Trento, Italy} (2009)

\bibitem{rasmussen_gaussian_2006}
Rasmussen, C.E., Williams, C.K.I.: Gaussian Processes for Machine Learning.
\newblock Adaptive Computation and Machine Learning. {MIT Press}, {Cambridge,
  Mass} (2006)

\bibitem{schulz_tutorial_2018}
Schulz, E., Speekenbrink, M., Krause, A.: A tutorial on {{Gaussian}} process
  regression: {{Modelling}}, exploring, and exploiting functions.
\newblock Journal of Mathematical Psychology \textbf{85}, 1--16 (2018).
\newblock \doi{10.1016/j.jmp.2018.03.001}

\bibitem{shahriari_taking_2016}
Shahriari, B., Swersky, K., Wang, Z., Adams, R.P., {de Freitas}, N.: Taking the
  {{Human Out}} of the {{Loop}}: {{A Review}} of {{Bayesian Optimization}}.
\newblock Proceedings of the IEEE \textbf{104}(1), 148--175 (2016).
\newblock \doi{10.1109/JPROC.2015.2494218}

\bibitem{toal_development_2011}
Toal, D.J., Bressloff, N.W., Keane, A.J., Holden, C.M.: The development of a
  hybridized particle swarm for kriging hyperparameter tuning.
\newblock Engineering Optimization \textbf{43}(6), 675--699 (2011).
\newblock \doi{10.1080/0305215X.2010.508524}

\bibitem{toal_kriging_2008}
Toal, D.J.J., Bressloff, N.W., Keane, A.J.: Kriging {{Hyperparameter Tuning
  Strategies}}.
\newblock AIAA Journal \textbf{46}(5), 1240--1252 (2008).
\newblock \doi{10.2514/1.34822}

\bibitem{toal_adjoint_2009}
Toal, D.J.J., Forrester, A.I.J., Bressloff, N.W., Keane, A.J., Holden, C.: An
  adjoint for likelihood maximization.
\newblock Proceedings of the Royal Society A: Mathematical, Physical and
  Engineering Sciences \textbf{465}(2111), 3267--3287 (2009).
\newblock \doi{10.1098/rspa.2009.0096}

\bibitem{ulaganathan_performance_2016}
Ulaganathan, S., Couckuyt, I., Dhaene, T., Degroote, J., Laermans, E.:
  Performance study of gradient-enhanced {{Kriging}}.
\newblock Engineering with Computers \textbf{32}(1), 15--34 (2016).
\newblock \doi{10.1007/s00366-015-0397-y}

\bibitem{won_maximum_2006}
Won, J.H., Kim, S.J.: Maximum {{Likelihood Covariance Estimation}} with a
  {{Condition Number Constraint}}.
\newblock In: 2006 {{Fortieth Asilomar Conference}} on {{Signals}}, {{Systems}}
  and {{Computers}}, pp. 1445--1449. {IEEE}, {Pacific Grove, CA, USA} (2006).
\newblock \doi{10.1109/ACSSC.2006.354997}

\bibitem{wu_exploiting_2018}
Wu, A., Aoi, M.C., Pillow, J.W.: Exploiting gradients and {{Hessians}} in
  {{Bayesian}} optimization and {{Bayesian}} quadrature.
\newblock arXiv:1704.00060 [stat]  (2018)

\bibitem{wu_bayesian_2017}
Wu, J., Poloczek, M., Wilson, A.G., Frazier, P.: Bayesian {{Optimization}} with
  {{Gradients}}.
\newblock In: 31st {{Conference}} on {{Neural Information Processing Systems}}.
  {Long Beach, CA, USA} (2017)

\bibitem{zimmermann_maximum_2013}
Zimmermann, R.: On the {{Maximum Likelihood Training}} of {{Gradient-Enhanced
  Spatial Gaussian Processes}}.
\newblock SIAM Journal on Scientific Computing \textbf{35}(6), A2554--A2574
  (2013).
\newblock \doi{10.1137/13092229X}

\bibitem{zimmermann_condition_2015}
Zimmermann, R.: On the condition number anomaly of {{Gaussian}} correlation
  matrices.
\newblock Linear Algebra and its Applications \textbf{466}, 512--526 (2015).
\newblock \doi{10.1016/j.laa.2014.10.038}

\bibitem{zingg_comparative_2008}
Zingg, D.W., Nemec, M., Pulliam, T.H.: A comparative evaluation of genetic and
  gradient-based algorithms applied to aerodynamic optimization.
\newblock European Journal of Computational Mechanics \textbf{17}(1-2),
  103--126 (2008).
\newblock \doi{10.3166/remn.17.103-126}

\end{thebibliography}

%% Non-BibTeX users please use
%\begin{thebibliography}{}
%%
%% and use \bibitem to create references. Consult the Instructions
%% for authors for reference list style.
%\end{thebibliography}

\end{document}